\documentclass[10pt]{amsart}
\long\def\drop#1{}

\usepackage{a4wide}
\usepackage[english]{babel}
\usepackage{graphicx}
\usepackage{enumerate}
\usepackage{pict2e}
\usepackage{amsmath,amsthm}
\usepackage{amsfonts}
\usepackage{amssymb}
\usepackage{color}
\usepackage{url}

\DeclareMathOperator{\dil}{dil}
\DeclareMathOperator{\Vol}{Vol}

\DeclareMathOperator\Ric{Ric}
\DeclareMathOperator\inj{inj}
\DeclareMathOperator\Hess{Hess}

\def\R{{\mathbb R}}

\newtheorem{Theorem}{Theorem}[section]
\newtheorem{Lemma}[Theorem]{Lemma}

\numberwithin{equation}{section}

\title{Embeddings of Riemannian manifolds with heat kernels and eigenfunctions}
\author{Jacobus W. Portegies}

\begin{document}

\begin{abstract}
We show that any closed $n$-dimensional Riemannian manifold can be embedded by a map constructed from heat kernels at a certain time from a finite number of points. Both this time and this number can be bounded in terms of the dimension, a lower bound on the Ricci curvature, the injectivity radius and the volume. It follows that the manifold can be embedded by a finite number of eigenfunctions of the Laplace operator. Again, this number only depends on the geometric bounds and the dimension. In addition, both maps can be made arbitrarily close to an isometry. In the appendix, we derive quantitative estimates of the harmonic radius, so that the estimates on the number of eigenfunctions or heat kernels needed can be made quantitative as well.
\end{abstract}

\maketitle

\section{Introduction}

\subsection{Eigenmaps and Diffusion Maps}

Measurements in large experiments or large amounts of data collected for the purpose of machine learning often satisfy certain nonlinear constraints, that is, the data lies on a submanifold of the space of all possible outcomes. Due to the linear character of classical dimension reduction methods such as Principal Component Analysis and classical Multidimensional Scaling, they are ill-suited to pick up the nonlinear structure. In recent years, several nonlinear alternatives have been developed \cite{scholkopf_kernel_1997,roweis_nonlinear_2000,belkin_laplacian_2003,
donoho_hessian_2003,zhang_principal_2004,coifman_diffusion_2006,singer_vector_2012}. 

Two of these algorithms, namely the methods of Eigenmaps \cite{belkin_laplacian_2003} and Diffusion Maps \cite{coifman_diffusion_2006}, make use of eigenfunctions of the Laplace operator on graph approximations of the underlying manifold to embed the manifolds in a lower-dimensional Euclidean space.

The basic idea is as follows. Suppose $\phi_k$ are the eigenfunctions of the (negative of the) Laplace operator on an underlying manifold $M$, that is $- \Delta \phi_k = \lambda_k \phi_k$, where $\lambda_0 < \lambda_1 \leq \lambda_2 \leq \dots$. Loosely speaking, the Diffusion Maps and Eigenmaps are variants of the map $F_N$, defined by
\begin{equation}
F_N(p) := ( \phi_1(p), \dots, \phi_N(p) ).
\end{equation}
Experimentally, it turns out that such a map gives a useful representation or embedding of the manifold. Diffusion Maps have been applied succesfully to problems in machine learning, see for instance \cite{liu_learning_2009}, and to recognize patterns in experiments \cite{giannakis_nonlinear_2012}.

The question whether $F_N$ yields an embedding is also interesting from a different perspective. Whitney's Embedding Theorem states that any smooth $n$-dimensional manifold can be embedded into $\R^{2n}$. Although the statement that the manifold can be embedded in $\R^{2n}$ relies on a clever trick, it is much easier to prove that any $n$-dimensional manifold can be embedded into $\R^{2n+1}$. Greene and Wu \cite{greene_embedding_1975} showed that one can properly embed any noncompact, smooth $n$-dimensional Riemannian manifold into $\R^{2n+1}$ by using harmonic component functions. By the maximum principle, for a compact manifold an embedding with harmonic functions would be impossible. However, a next natural question is whether such an embedding can be performed using eigenfunctions. 

\subsection{Can a manifold be embedded with eigenfunctions?}

B\'{e}rard \cite{berard_volume_1985} answered this question positively, by showing that a given manifold $M$ can be embedded by a normalized version of $F_N$ in a high dimensional unit sphere. 
Later, B\'{e}rard, Besson and Gallot constructed an embedding using all eigenfunctions in the sequence space $\ell^2$ in \cite{berard_embedding_1994}. 
The proof uses that the eigenfunctions form a complete basis for $L^2(M)$, and hence separates points, and that for smooth functions the differential can be obtained by differentiation of the eigenfunction expansion. 
This suggests that in order to get a bound on the number of eigenfunctions needed, these statements about eigenfunction expansions should be quantified.

Abdallah \cite{abdallah_embedding_2012} showed that a time-varying manifold can be embedded by eigenfunctions, when the time-varying metric is analytic. The proof uses that in this setting the space of embeddings will be open, and that the embedding by B\'{e}rard et al. can therefore be truncated.

In 2008, Jones, Maggioni and Schul showed that any smooth $n$-dimensional manifold admitting charts in which the metric is $C^\alpha$, can be locally embedded in $\mathbb{R}^n$ by eigenfunctions of the Laplace Operator \cite{jones_manifold_2008}. They observe that this follows from using a technique that they call the heat triangulation: first they show that a small ball can be embedded by the values of the heat kernels from various points on the manifold. After showing that the same construction can be performed with a truncated heat kernel, they conclude that there are eigenfunctions that embed the small ball. In the extended version of the paper \cite{jones_universal_2010}, Jones et al. mention that a global embedding with eigenfunctions can be obtained as well, but they do not prove this statement. 

\subsection{Further results on embeddings}

Embeddings of this type are reminiscent of the Kodaira embedding theorem in K\"{a}hler geometry. This analogy is reflected in the works by Zelditch \cite{zelditch_fine_1997,zelditch_real_2009} and Potash \cite{potash_euclidean_2013}. Zelditch \cite{zelditch_fine_1997,zelditch_real_2009} has proved that for manifolds of a certain type (for which the geodesic flow is either periodic or aperiodic), the first eigenfunctions provide an almost isometric embedding in Euclidean space. Recently, Potash \cite{potash_euclidean_2013} explained how the argument extends for general compact manifolds.

The embedding by B\'{e}rard, Besson and Gallot \cite{berard_embedding_1994} is obtained by composing the map that sends a point $p$ on the manifold to its heat kernel $K(p,t;.)$, with the eigenfunction expansion. Nicolaescu \cite{nicolaescu_random_2012} adapted the construction by B\'{e}rard, Besson and Gallot, in that he replaced the heat kernel by a kernel that only depends on finitely many eigenfunctions, thereby also proving that one can embed by a finite number of eigenfunctions. 

The embeddings discussed so far are in general at best close to isometric. In contrast, Nash's embedding theorem shows the existence of an isometric embedding in Euclidean space. Recently, Wang and Zhu \cite{wang_isometric_2013} used the embedding by B\'{e}rard, Besson and Gallot in an iteration scheme to create a canonical family of \emph{isometric} embeddings. 

Singer and Wu \cite{singer_vector_2012} introduced a data analysis algorihm similar to the Eigenmaps and Diffusion Maps algorithms, that is however based on the connection Laplacian for vector fields, rather than the Laplace-Beltrami operator for functions. 
Analogous to the embedding by B\'{e}rard, Besson and Gallot, they construct an embedding in a sequence space $\ell^2$, this time recording the inner products of eigenfunctions of the connection Laplacian. 
In \cite{wu_embedding_2013} Wu uses this embedding to introduce a diffusion distance between manifolds and subsequently shows a precompactness theorem for this distance in a class of manifolds given certain geometric bounds.

\subsection{How many eigenfunctions does one need?}

An important question that remains is: How many eigenfunctions does one need to embed and accurately represent the manifold? This question is especially important from a computational perspective, since an answer may lead to analogous bounds on the number of eigenfunctions needed in the Eigenmaps \cite{belkin_laplacian_2003} and Diffusion Maps \cite{coifman_diffusion_2006} algorithms. 

It is clear that the number of eigenfunctions needed to embed the manifold will depend on its geometric complexity. Consider for instance a surface in $\R^3$ consisting of $N$ punctured spheres connected with very small tubes. It follows from the min-max principle and the Poincar\'{e} inequality that the first $N$ eigenfunctions are almost constant on each of the spheres. Indeed, it is easy to construct $N$ orthogonal functions that are constant on the spheres and have Rayleigh quotient close to zero. The first $N$ eigenvalues are therefore close to zero and the functions are $L^2$-close to their average on the spheres by the Poincar\'{e} inequality. Therefore, $F_N$ will not be an embedding. Another example is a flat torus $S^1 \times \epsilon S^1$, which is not embedded by $N$ eigenfunctions unless the corresponding eigenvalue $\lambda_N$ is larger than $1/\epsilon^2$.

As was also observed by Jones et al.~\cite{jones_manifold_2008}, in terms of analysis, it may be easier to think of embedding with heat kernels. Whereas they used heat kernels to obtain a local embedding, we would like to obtain a global embedding, by recording the heat kernels from many points on the manifold, reminiscent of a construction by Gromov \cite{gromov_structures_1981} in which he embeds manifolds into $\R^N$ endowed with the maximum norm, by using truncated distance functions. Let $K$ denote the heat kernel on a manifold $M$. Let us build a map that in its components records the value of the heat kernels based at certain points at a certain time. That is, for a finite set of points $\{q_i\}_{i=1}^{N_0} \subset M$, and a time $t$, define
\begin{equation}
\label{eq:DefMapG}
G(p) := (2t)^\frac{n+1}{2} \left( K(p,t;q_1), \dots, K(p,t;q_{N_0}) \right).
\end{equation}
The question whether $G$ yields an embedding is in fact interesting on its own.

Our objective is to bound the number of heat kernels and the number of eigenfunctions needed in terms of geometric information on the manifold. Moreover, we want to make sure that distances are approximately preserved under the embeddings, that is, we want the local dilatation to be close to one.

Both B\'{e}rard \cite{berard_volume_1985} and B\'{e}rard et al.~\cite{berard_embedding_1994} used the Minakshisundaram-Pleijel asymptotic expansion (cf. \cite{berger_spectre_1971}), that links the heat kernel to the geometry of the manifold. This expansion holds for a smooth manifold, and may not be used directly in less regular settings. In particular, we at most wanted to use $C^\alpha$ bounds on the metric, which led us to not use the expansion.

The local embedding by Jones et al. \cite{jones_manifold_2008} is bi-Lipschitz, with constants that depend on bounds in coordinates: their starting assumption is that there exists a coordinate patch on which the metric is H\"{o}lder continuous, bounded, and coercive. However, these bounds mean that at a smaller scale, and in possibly different coordinates, the manifold is actually close to Euclidean. It is at this scale that the construction with the heat kernels can be performed succesfully. The scale is more or less expressed in terms of the radius at which the coordinates exist and the bounds on the metric, but we would like to replace these conditions by ones that are more geometric.

\subsection{The scale at which a manifold looks Euclidean: the harmonic radius}

Naturally, the scale at which the manifold looks Euclidean plays a large role and is related to the curvature. In case we would like to get bounds on the dilatation, and bounds on the number of eigenfunctions needed, we would need to use coordinates with good regularity. It follows by the work of DeTurck and Kazdan \cite{deturck_regularity_1981} that harmonic coordinates have optimal regularity properties. 

Estimates on the harmonic radius provide a passage from geometric information to a scale at which the metric is close to Euclidean. 
A typical harmonic radius estimate assumes bounds on the geometry of a manifold, for instance on the curvature, diameter and injectivity radius, and concludes that on balls with radius smaller than the harmonic radius, there exist harmonic coordinates in which the metric is close to Euclidean. 
What sense can be given to `close', depends on the exact assumptions on the geometry. 
Hebey and Herzlich's survey \cite{hebey_harmonic_1997} nicely sums up various results.

Here we would like to highlight a few estimates. Jost and Karcher \cite{jost_geometrische_1982} obtain a bound under uniform control on the sectional curvature. Anderson \cite{anderson_convergence_1990} shows that under uniform bounds on the Ricci curvature, the metric can be controlled in the $C^{1,\alpha}$ sense. 

We will initially use the estimate on the $C^\alpha$ harmonic radius by Anderson and Cheeger \cite{anderson_compactness_1992}, as for our purposes it suffices to have $C^\alpha$ control of the metric. Anderson and Cheeger show that in the class of manifolds with a fixed lower bound on the curvature and injectivity radius, there is a uniform lower bound for the harmonic radius, that is, on any ball with smaller radius there exist harmonic coordinates for which the metric coefficients are close to Euclidean and have a small $C^\alpha$ norm for any $0 < \alpha < 1$. 

Anderson and Cheeger obtained the existence of the harmonic radius by a compactness argument, that does not give a quantitative estimate. 
From the perspective of applications to data analysis, it may be important to be able to make quantitative statements.
The harmonic radius estimate by Jost and Karcher \cite{jost_geometrische_1982} is quantitative, but assumes bounds on the sectional curvature, rather than only a lower bound on the Ricci curvature. 
However, it is possible to combine some estimates in the work by Anderson and Cheeger \cite{anderson_compactness_1992} with the Bishop-Gromov inequality and the segment inequality to obtain H\"{o}lder continuity of the metric in distance function coordinates. 
From there, one may solve a Dirichlet problem to obtain harmonic coordinates with quantitative estimates. We present the argument in the appendix.

After initial submission of the manuscript, the author has learned about a recent paper by Bates \cite{bates_embedding_2014} that combines the harmonic radius estimate with the result by Jones et al. \cite{jones_manifold_2008} to obtain a bound on the number of eigenfunctions needed to embed a manifold. Our work differs in that we do not rely directly on the work of Jones et al., but rather apply PDE arguments to arrive at similar results. Moreover, we additionally show that the embeddings are close to isometric.

\subsection{Summary of main results}

In Theorem \ref{Th:EmbHeatKernInf} we show that for small enough time $t$ and a dense enough net $\{q_i\}$ the map $G$ as defined in (\ref{eq:DefMapG}) is indeed an embedding. Additionally, we show that after normalization, the map is almost an isometry when $\R^{N_0}$ is endowed with the maximum norm. 

However dense a net $\{q_i\}$ may be, the map $G$ may not be almost an isometry when embedding into $\R^{N_0}$ with the Euclidean norm, since the measure on the manifold starts playing a role.  In Theorem \ref{Th:EmbDiffWeight} we show that we can weigh the different components differently, and obtain an almost isometry in Euclidean space.
A simple argument then shows that if the points $q_i$ can be chosen more specifically, the map $G$ is an embedding in $\R^{N_0}$ that is almost an isometry. The result is presented in Theorem \ref{Th:EmbHeatKernEucl}.

We estimate the time from below, and the number of points needed from above in terms of the dimension, a Ricci curvature lower bound, the injectivity radius and the volume of the manifold.

In both theorems, we can replace the heat kernel $K$ by the truncated version $K_N$ (as defined in (\ref{eq:HeatKernTrunc})), and the statements still hold. It follows immediately that a manifold $M$ can be embedded using a finite number of eigenfunctions, only depending on the dimension, the Ricci curvature and injectivity radius lower bounds and the volume upper bound. However, in Theorem \ref{Th:EmbEigFunc} we show that a suitable multiple of the map
\begin{equation}
\mathcal{F}_N(p) := (2t)^{\frac{n+2}{4}} \sqrt{2}(4\pi)^{n/4} 
\left(e^{-\lambda_1 t} \phi_1(p) , \dots, e^{- \lambda_N t} \phi_N(p) \right),
\end{equation}
can be made approximately an isometry by taking $t$ small enough, and $N$ large enough. This map is in fact a truncation of the map constructed by B\'{e}rard et al. \cite{berard_embedding_1994}, that was also shown to asymptotically preserve the metric. 
Abdallah has shown, in the time-varying case, that the truncation is asymptotically isometric \cite[Theorem 1.3 (ii)]{abdallah_embedding_2012}. 
The important point is that we can truncate the map uniformly with respect to the geometric bounds and the dimension, and that we do not rely on smoothness of the manifold.

We believe that the use of the harmonic radius makes the argument very transparent. Moreover, Theorems \ref{Th:EmbHeatKernInf} and \ref{Th:EmbDiffWeight}, about embeddings in $\R^{N_0}$ endowed with the maximum norm and Euclidean norm respectively, provide a unifying perspective on a variant of the Kuratowski embedding, or the construction by Gromov \cite{gromov_structures_1981}, the local embedding by Jones et al. \cite{jones_manifold_2008}, and the embeddings by B\'{e}rard \cite{berard_volume_1985} and B\'{e}rard et al. \cite{berard_embedding_1994}. Indeed, for small times, distance functions and the heat kernel are closely related, as is illustrated already by Varadhan's asymptotic formula derived in \cite{varadhan_diffusion_1967}
\[
\lim_{t\downarrow 0} 4t \log K(p,t;q) = - d(p,q)^2.
\]

The organization of the paper is as follows. In Section \ref{se:HeatKernelEstimates}, we will reference estimates on the decay of the heat kernel and its gradient, and bounds on eigenfunctions and eigenvalues that help us show that we can truncate the heat kernel.
In Section \ref{se:HeatTriangulation}, we will prove Theorems \ref{Th:EmbHeatKernInf}, \ref{Th:EmbDiffWeight}, and \ref{Th:EmbHeatKernEucl} that show existence of the embedding using the heat kernel. We will use these results in Section \ref{se:EmbEig} to construct the embedding with eigenfunctions. In the proofs in Section \ref{se:HeatTriangulation}, we will need to use that on a small enough scale, the heat kernel is actually close to the Euclidean version. For that, we will use parabolic Schauder estimates as presented in Section \ref{se:SchauderTheory}. Finally, we derive quantitative estimates on the harmonic radius in the appendix.

\subsection*{Acknowledgments}

I would like to thank my thesis advisor Fanghua Lin for suggesting this problem to me and for helpful discussions. Moreover,  I would like to thank Jeff Cheeger for discussions about the quantitative estimates of the harmonic radius. 

\section{Some background and notation}

Throughout the paper, we will use the letter $C$ for a constant that may change from step to step in a computation.

We denote by $\mathcal{M}(n,\kappa,\iota,V)$ the set of $n$-dimensional, closed Riemannian manifolds $(M,g)$ with the volume bounded above by $V$, Ricci curvature bounded below by $\kappa$, and injectivity radius bounded below by $\iota$. For simplicity, we assume that $M$ and $g$ are smooth, although this assumption can be weakened. 

We denote by $K(p,t;q)$ the heat kernel on $M$. By definition, the heat kernel satisfies for $p,q \in M$, $t > 0$,
\[
\partial_t K(p,t;q) - \Delta K(p,t;q) = 0.
\]
Moreover, for every continuous function $f$ on $M$,
\[
\lim_{t\downarrow 0} \int_M K(p,t;q) f(p) dp = f(q).
\]
The heat kernel on a manifold has the following representation
\[
K(p,t;q) = \sum_{k=0}^\infty e^{- \lambda_k t } \phi_k(p) \phi_k(q),
\]
where $\phi_k$ are the eigenfunctions of the Laplace operator, $-\Delta \phi_k = \lambda_k \phi_k$, normalized by $|\phi_k|_2 = 1$. We define the truncated heat kernel $K_N$ by
\begin{equation}
\label{eq:HeatKernTrunc}
K_N(p,t;q) := \sum_{k=0}^N e^{ - \lambda_k t } \phi_k(p) \phi_k(q).
\end{equation}
Finally, we denote by $\Gamma_E$ the standard heat kernel in $\R^n$.
\[
\Gamma_E(x,t;y)
:= \frac{1}{(4 \pi t)^{n/2}} 
	\exp\left[ - \frac{|x-y|^2}{4 t} \right].
\]

The local dilatation of a map $f$ from a metric space $X$ to a metric space $Y$ at a point $p$ is defined as
\begin{equation}
\dil_p(f) := \lim_{r \to 0} \sup_{x,y \in B_r(p) } \frac{d(f(x),f(y))}{d(x,y)}.
\end{equation}
When $M$ is a smooth Riemannian manifold, that is embedded by a smooth map $f$ into a normed, finite-dimensional vector space $V$, the dilatation is given by
\[
\dil_p(f) = |(d f)_p|,
\]
where the norm on the right hand side is interpreted as the operator norm of the map from $T_p(M)$ to $T_{f(p)}(f(M))$. 

\subsection{The harmonic radius}
\label{se:HarRad}

With a lower bound on the Ricci curvature and injectivity radius, there is a lower bound on the radius of balls on which there exist harmonic coordinates. This radius will determine the scale that will play an important role in the rest of the paper. Anderson and Cheeger proved the following theorem \cite{anderson_compactness_1992}. 

\begin{Theorem}[cf. {\cite{anderson_compactness_1992}}]
For every $Q > 1$ and $0 < \alpha < 1$, there is a radius $r_h(n,\kappa,\iota,\alpha,Q)$ such that for every $(M,g) \in \mathcal{M}(n,\kappa,\iota,V)$, and any ball $B_r(p)$ on $M$ with $r \leq r_h$, there exist harmonic coordinates $u: B_r(p) \to \mathbb{R}^n$ such that the coefficients $g_{ij}$, given by,
\begin{equation}
g_{ij} = g\left( \frac{\partial}{\partial u^i}, \frac{\partial}{\partial u^j} \right),
\end{equation}
satisfy

\begin{subequations}
\label{eq:BoundsOng}
\begin{align}
Q^{-1} (\delta_{ij}) \leq (g_{ij}) &\leq Q (\delta_{ij}) \qquad \text{ as bilinear forms},\\
r_h^\alpha \| g_{ij} \|_{C^\alpha} &\leq Q - 1.
\end{align}
\end{subequations}
\end{Theorem}
In the appendix we will give a quantitative estimate on the harmonic radius $r_h$.

\section{Some heat kernel estimates}
\label{se:HeatKernelEstimates}

We will now recall some properties of the decay of the heat kernel, following mostly the book by Grigor'yan \cite{grigoryan_heat_2012}. We will show how these estimates imply heat kernel decay in coordinates. Moreover, we will see how the decay of the heat kernel implies growth of the eigenvalues of the Laplace operator on the manifold, with a lower bound expressed in the Ricci curvature, volume, injectivity radius, and the dimension. After combining this with elliptic estimates, we conclude that the heat kernel can be truncated.

\subsection{Heat kernel decay}

By the estimate on the harmonic radius in Section \ref{se:HarRad}, there is a radius $r_h = r_h(n,\kappa,\iota,\alpha=1/2,Q=\sqrt{2})$ such that for every open subset $U \subset M$ contained in a ball with radius less than $r_h$, the Faber-Krahn inequality holds for $\lambda_{\min}(U)$, the smallest Dirichlet eigenvalue of the Laplace operator on the domain $U$,
\begin{equation}
\lambda_{\min} (U) \geq a(n) |U|^{-2/n},
\end{equation}
where for any measurable set $U$ on the manifold, $|U|$ denotes the standard volume measure of $U$.
By \cite[Theorem 15.14]{grigoryan_heat_2012}, 
\begin{equation}
\label{eq:DecHeatKernMan}
K(p,t;q) \leq \frac{C(n) \left( 1+ \frac{d(p,q)^2}{t} \right)^{n/2}}
 {\left(a(n) \min(t,r_h^2)\right)^{n/2}}\exp\left[ - \frac{d(p,q)^2}{4t}\right],
\end{equation}
where $d(p,q)$ denotes the (geodesic) distance between $p$ and $q$. 

From interior parabolic Schauder estimates it also follows that for $t \leq 2 r_h^2$,
\begin{equation}
\label{eq:GradDecHeatKernMan}
|\nabla K(p,t;q)| \leq \frac{D(n)}
 {t^{(n+1)/2}}\exp\left[ - \frac{d(p,q)^2}{8t}\right].
\end{equation}
Indeed, for points $p$ and $q$, we can use parabolic interior Schauder estimates (cf. \cite[Ch. 4, Theorem 4]{friedman_partial_2008}) on a ball around $p$, to conclude (\ref{eq:GradDecHeatKernMan}) for $t \leq \min ( d(p,q)^2/2 , 2 r_h^2)$. 
For $t > d(p,q)^2/2$, we can use that the heat kernel is $C^1$-close to the Euclidean heat kernel, as explained in Section \ref{se:SchauderTheory}, to conclude that the bound (\ref{eq:GradDecHeatKernMan}) also holds on this scale.

\subsection{Heat kernel decay in coordinates}
\label{se:DecCoord}

Let $p \in M$, let $Q \leq \sqrt{2}$ and  $r < r_h(n,\kappa,\iota,\alpha = 1/2, Q)$, and let the coordinates $u:B_{r_h} \to \R^n$ be harmonic satisfying (\ref{eq:BoundsOng}) and $u(p) = 0$. 
Define the rescaled heat kernels
\begin{equation}
\tilde{K}(x,s;q) := r^n K\left( u^{-1}(x r), s r^2; q \right),
\end{equation}
and
\begin{equation}
\Gamma(x,s;y) := r^n K \left(u^{-1}(xr), s r^2; u^{-1}(y r) \right).
\end{equation}

It follows that in this case, there is a constant $C_d = C_d(n)$ such that for $s < 2 r_h^2/r^2$,
\begin{equation}
\label{eq:DecCoord}
\Gamma(x,s;y) \leq \frac{C_d(n)}{ s^{n/2} } \exp\left[ - \frac{|x-y|^2}{8s}\right],
\end{equation}
and a constant $D_d= D_d(n)$ such that
\begin{equation}
\label{eq:GradDecCoord}
\left| \nabla \Gamma(x,s;y) \right| \leq 
	\frac{D_d(n)}{s^{(n+1)/2}} \exp\left[ - \frac{|x-y|^2}{8s} \right].
\end{equation}

\subsection{Eigenvalue growth}

We use (\ref{eq:DecHeatKernMan}) to  bound the trace of the heat kernel as follows
\begin{equation}
\int_M K(p,t;p) dp \leq \mathrm{Vol}(M) \frac{C(n)}{(a(n) \min(t,r_h^2))^{n/2}}.
\end{equation}
It follows by \cite[Theorem 14.25]{grigoryan_heat_2012} that if
\begin{equation}
k \geq \frac{C(n) \mathrm{Vol}(M)}{a(n)^{n/2} r_h^n}e^{n/2},
\end{equation}
then the following lower bound on $\lambda_k$ holds
\begin{equation}
\label{eq:EigValLowBound}
\lambda_k(M) \geq \frac{n}{2e} \, a(n) \left( \frac{k}{C(n) \mathrm{Vol}(M)} \right)^{2/n}.
\end{equation}

\subsection{Bounds on the eigenfunctions and their derivatives}

In the following lemma, we use elliptic estimates to get bounds on the supremum norm of the eigenfunctions and their gradients in terms of their $L^2$ norm. These bounds follow from local arguments, and while not optimal from a global perspective, they are good enough for our purposes (cf. \cite{donnelly_eigenfunctions_2006}). 

\begin{Lemma}
\label{le:EigFuncBounds}
There is a constant $C = C(n,\kappa,\iota)$ such that for all $(M,g) \in \mathcal{M}(n,\kappa,\iota,V)$ and eigenfunctions $\phi_k$ of the (negative of the) Laplace operator on $M$, with corresponding eigenvalues $\lambda_k$, it holds that for $k \geq k(n,\kappa,\iota,V)$,
\begin{subequations}
\label{eq:EigFuncBounds}
\begin{align}
\|\phi_k\|_\infty &\leq C \lambda_k^{n/4} \| \phi_k \|_2, \\
\|\nabla \phi_k\|_\infty &\leq C \lambda_k^{(n+2)/4} \| \phi_k \|_2.
\end{align}
\end{subequations}
\end{Lemma}

\begin{proof}
Let $r_h=r_h(n,\kappa,\iota,\alpha=1/2,Q=2)$ be the harmonic radius. 
Let $p\in M$. 
Select harmonic coordinates $u:B_{r_h}(p)\to \R^n$ so that $u(p)=0$ and the metric coefficients $g^{ij}$ with respect to these coordinates satisfy (\ref{eq:BoundsOng}). 
The eigenfunctions $\phi_k$ satisfy $- \Delta \phi_k = \lambda_k \phi_k$ on the manifold. 
By the estimate (\ref{eq:EigValLowBound}) on the growth of the eigenvalues $\lambda_k$ we may now pick $k$ large enough, depending only on $n, \kappa, \iota$ and $V$, such that $\lambda_k \geq 1/r_h^2$.
We introduce coordinates $x = u \sqrt{\lambda_k}$, and write down the equation for $\phi_k$
\[
g^{ij}(x / \sqrt{\lambda_k}) \partial_{x^i}\partial_{x^j} \phi_k = \phi_k, \qquad x \in u(B_{r_h}(p))/r.
\]
Note that $B_{\frac{1}{2}\sqrt{2}}(0) \subset u(B_r(p))$.
Since the equation has bounded coefficients, if $|x| \leq 1/2$,
\[
| \phi_k(x) | \leq C(n) \left( \int_{B_{\frac{1}{2}\sqrt{2}}(0)} |\phi_k(y)|^2 dy \right)^{1/2}.
\]
Consequently, by the elliptic Schauder estimates, for $|x| \leq 1/4$ also
\[
| \nabla \phi_k (x)| \leq C(n) \left( \int_{B_{\frac{1}{2}\sqrt{2}}(0)} |\phi_k(y)|^2 dy \right)^{1/2}.
\]
This implies (\ref{eq:EigFuncBounds}).
\end{proof}

\subsection{Truncation of the heat kernel}

Using the bounds on the eigenfunctions derived in the previous section, we can control the tail of the heat kernel.

\begin{Lemma}
\label{le:TruncHeatKernClose}
Let $M \in \mathcal{M}(n,\kappa,\iota,V)$. Let $\epsilon > 0$ and $t_0>0$ be given. Then there exists $N_0 = N_0(n,\kappa,\iota,V,\epsilon,t_0)$, such that when $N \geq N_0$, for every $t_0 \leq t \leq 4$,
\begin{align}
\label{eq:TruncHeatKernClose}
\| K_N(.,t;q) - K(.,t;q) \|_\infty &< \epsilon,\\
\label{eq:DerTruncHeatKernClose}
\|\nabla K_N(.,t;q) - \nabla K(.,t;q) \|_\infty &< \epsilon.
\end{align}
\end{Lemma}

\begin{proof}
Consider the sum
\[
K_{N_1}^{N_2}(p,t;q) := \sum_{k=N_1}^{N_2} e^{-\lambda_k t} \phi_k(p) \phi_k(q).
\]
By the bounds (\ref{eq:EigFuncBounds}) we find that for a constant $C = C(n,\kappa,\iota)$, and $N_1 \geq k(n,\kappa,\iota,V)$,
\[
\left| \nabla K_{N_1}^{N_2}(p,t;q) \right|
\leq \sum_{k=N_1}^{N_2} e^{-\lambda_k t} | \nabla \phi_k(p)||\phi_k(q)|
\leq C \sum_{k=N_1}^{N_2} e^{-\lambda_k t} \lambda_k^{\frac{n+1}{2}}.
\]
Since the eigenvalues are bounded below as in (\ref{eq:EigValLowBound}), for $k\geq k_0(n,\kappa,\iota,V,t)$,
\[
e^{-\lambda_k t} \lambda_k^{\frac{n+1}{2}} \leq e^{-\lambda_k t/2}.
\]
With (\ref{eq:EigValLowBound}), we know that with a constant $c=c(n,\kappa,\iota,V)$,
\[
\sum_{k=N_1}^{N_2} e^{-\lambda_k t/2} \leq \sum_{k=N_1}^{N_2} e^{- c \, k^{2/n} t},
\]
and consequently, there is an $N_0 = N_0(n,\kappa,\iota,V,\epsilon,t_0)$ such that if $N_1 \geq N_0$ then (\ref{eq:DerTruncHeatKernClose}) holds. A similar argument shows that (\ref{eq:TruncHeatKernClose}) holds as well.
\end{proof}

\section{Embedding with heat kernels}
\label{se:HeatTriangulation}

In this section we will prove that manifolds can be embedded with heat kernels. In subsections \ref{suse:maximum} and \ref{suse:euclidean} we will show how the local dilatation can be controlled in case of an embedding into $\R^N$ endowed with the maximum norm and Euclidean norm respectively.

\subsection{Embedding with heat kernels in $\R^N$ with maximum norm}

The next theorem shows that the map $G$ is an embedding, and almost an isometry, for a dense enough net $\{q_i\}_{i=1}^{N_0}$, when the image space $\R^{N_0}$ is endowed with the maximum norm.
\label{suse:maximum}
\begin{Theorem}
\label{Th:EmbHeatKernInf}
For each $\epsilon > 0$, there is a $t_0 = t_0(n,\kappa,\iota,\epsilon)$ such that for all $0< t \leq t_0$ there is a $\delta = \delta(n,\kappa,\iota,\epsilon, t) > 0$ such that for all $M \in \mathcal{M}(n,\kappa,\iota,V)$ and every $\delta$-net $\{q_i\}_{i=1}^{N_0}\subset M$, the map $G$ given by
\begin{equation}
\begin{split}
G(p) &:= 
	(2t)^\frac{n+1}{2} ( K(p,t;q_1), \dots, K(p,t;q_{N_0}) ),
\end{split}
\end{equation}
is an embedding of $M$ into $(\R^{N_0},|.|_\infty)$, such that
\begin{equation}
1- \epsilon < (2\pi)^{\frac{n}{2}}e^{\frac{1}{2}}  |(dG)_p| < 1 + \epsilon.
\end{equation}
In addition, there exists an $N_E= N_E(n,\kappa,\iota,V,\epsilon,t)$ such that whenever $N\geq N_E$, the same statements hold for the map $G^N$, defined by
\begin{equation}
G^N(p) = (2t)^\frac{n+1}{2} (K_N(p,t;q_1), \dots, K_N(p,t;q_{N_0})).
\end{equation}
\end{Theorem}

The proof will be divided into three different steps. Step 1 is quite technical, and is mainly to clarify what the various constants depend on. In it, we select a scale $r = \sqrt{2t}$, and a large factor $\tilde{R}$, such that:
\begin{itemize}
\item On $B_{r\tilde{R}}(p)$, the heat kernel is close enough to the Euclidean heat kernel by results from Section \ref{se:SchauderTheory};
\item Fundamental solutions from points outside $B_{r \tilde{R}}(p)$ have very small gradients, so that they do not influence $|(dG)_p|$;
\item Fundamental solutions from points inside $B_{2r}(p)$ are small outside $B_{r\tilde{R}}(p)$, which we need to show that $G$ is one-to-one. 
\end{itemize}
With Euclidean heat kernels, we can do calculations explicitly, and it is clear how fundamental solutions yield an embedding. We show in Step 2 how to use the $C^1$-closeness to conclude that $G$ is an embedding. In Step 3 we estimate $|(dG)_p|$.

\begin{proof}

\textbf{Step 1: find an appropriate scale $r>0$.}

\noindent Let $\epsilon > 0$ be given. Set $\alpha = 1/2$. Pick $1 < Q < \sqrt{2}$ close enough to $1$, less than $Q_1(n,\epsilon)$ below, and such that 
\begin{equation}
2(Q-1)C(n,\alpha) \leq \sigma,
\end{equation}
where $C(n,\alpha)$ is as in Lemma \ref{le:CloseToEucl}, and $\sigma$ is less than $\sigma_1(n)$, $\sigma_2(n)$ and $\sigma_3(n,\epsilon)$, that are specified later.
Pick also $R_0=R(n,\alpha,C_d(n), Q)$ as in Lemma \ref{le:CloseToEucl}, where $C_d(n)$ is defined in (\ref{eq:DecCoord}).

Choose $R_1 = R_1(n)$ such that for $\Gamma$ satisfying the decay in (\ref{eq:DecCoord}), and $|y| < 2$, $|x| > R_1$, and $1/2\leq s \leq 2$,
\begin{equation}
\label{eq:ChoiceR1}
\Gamma(x,s;y) \leq \frac{1}{10} \inf_{\substack{z \in B_2 \\ 1/2\leq \tau \leq 2}} \Gamma_E(z,\tau;e_1),
\end{equation}
where $e_i$ $(i=1,\dots,n)$ stands for the $i$th standard unit vector in $\R^n$, with a $1$ in the $i$th coordinate and zeros in the remaining.

Let $R_2 = R_2(n)$ be large enough such that for every $\tilde{r} < r_h = r_h(n,\kappa,\iota,\alpha = 1/2, Q)$, if $d(p,q) > R_2 \tilde{r}$,
\begin{equation}
\label{eq:GradEstimate}
\tilde{r}^{n+1} \left| \nabla K\left(p, \tfrac{\tilde{r}^2}{2}; q \right)\right| < \frac{1}{2} (2 \pi)^{-\frac{n}{2}} e^{-\frac{1}{2}}.
\end{equation}

Define the scale $r_0$ by
\begin{equation}
r_0 := \frac{r_h(n,\kappa,\iota,\alpha = 1/2, Q) }{2 \max(R_0,R_1,R_2)},
\end{equation}
and set $t_0 := r_0^2/2$. Set $r \leq r_0$, and $t = r^2 /2$.\\

\noindent \textbf{Step 2: Show that the map is an embedding.}

\noindent Now suppose $q_1, \dots, q_{N_0}$ is a $\delta r$-net on $M$, where $\delta$ is less than $\delta_1(n)$, $\delta_2(n)$, and $\delta_3(n,\epsilon)$ specified later. Let $p \in M$, and choose harmonic coordinates $u: B_{r_h}(p) \to \R^n$, that satisfy (\ref{eq:BoundsOng}) and $u(p) = 0$. 
Denote $y_i := r^{-1} u(q_i)$ whenever $q_i \in B_{r_h}(p)$. 
Note that $B_{r_h/\sqrt{Q}} \subset u(B_{r_h}(p))$.

Introduce the rescaled heat kernel
\[
\tilde{K}( x, s; q ) := r^n K( u^{-1} (x r), s r^2; q ),
\]
and
\[
\Gamma(x, s; y ) := r^n K( u^{-1} (x r), s r^2; u^{-1}(y r) ).
\]

We will want to use the results on fundamental solutions of parabolic equations as presented in Section \ref{se:SchauderTheory}. Note that, because the coordinates $u$ are harmonic, $\Gamma$ is a fundamental solution to the operator $L$, defined by
\[
L u = u_s - a^{ij} \partial_{x^i}\partial_{x^j} u,
\]
on the domain $B_R(0)$, where
\[
R = \frac{2}{\sqrt{Q}} \max(R_0, R_1, R_2),
\]
and $a^{ij}(x) = g^{ij}(r x)$, so that the coefficients $a^{ij}$ satisfy (\ref{eq:BoundsOna}). 
Moreover, as explained in Section \ref{se:DecCoord}, $\Gamma$ satisfies the exponential decay estimate (\ref{eq:DecCoord}). 
By Lemma \ref{le:CloseToEucl} and our choice of $Q$ and $R \geq R_0$, we know that for $y \in B_R$, $x \in B_2$,
\begin{equation}
\label{eq:GammaCloseGammaE}
|\nabla \Gamma(x,s;y) - \nabla \Gamma_E(x,s;y) | < \sigma,
\end{equation}
for $1/2 \leq s \leq 2$. Moreover, for $y \in B_2$ and $x \in B_R$,
\begin{equation}
|\Gamma(x,s;y) - \Gamma_E(x,s;y) | < \sigma.
\end{equation}

We calculate
\[
\nabla \Gamma_E(x,s;y) = - \frac{1}{2s} (x - y) \Gamma_E(x,s;y).
\]
Since the $q_i$ form a $(\delta r)$-net, for every $j = 1, \dots, n$, there is an index $i_j$ such that $|y_{i_j} - e_j| \leq 2\delta$, where $e_j$ is the $j$th standard unit vector in $\R^n$. 
Then, when $|x|<4\rho$,
\[
\left| \nabla \Gamma_E(x,s;y_{i_j}) - \frac{1}{2s} e_j \Gamma_E(0,s;e_1) \right| \leq C(n) (\delta + \rho).
\]
Hence, also
\begin{equation}
\label{eq:CloseToUnitVec}
\left| \nabla \Gamma(x,s;y_{i_j}) - \frac{1}{2s} e_j \Gamma_E(0,s;e_1)\right| \leq C(n) (\delta + \rho) + \sigma.
\end{equation}

Therefore, when $\sigma < \sigma_1(n)$, $\delta < \delta_1(n)$, and $\rho < \rho_1(n)$, the ball $B_{4\rho}$ is embedded by the map
\begin{equation}
x \mapsto \left( \Gamma(x,t; y_{i_1}) , \dots, \Gamma(x,t; y_{i_n} ) \right).
\end{equation}
Set
\begin{equation}
\Phi_E: x \mapsto \left(\Gamma_E(x,s;0), \Gamma_E(x,s;e_1), \Gamma_E(x,s;e_2), \dots, \Gamma_E(x,s;e_n) \right),
\end{equation}
and note that there is a distance $d_1(n) > 0$ between the sets $\{ \Phi_E(y) \, | \, y \in B_\rho(0) \}$ and $\{ \Phi_E(y) \, | \, y \in \R^n \backslash B_{2\rho}(0)  \}$ in $\R^n$. 
Let $y_{i_0}$ be such that $|y_{i_0} | < 2\delta$, and consider the map $\Phi$ given by
\begin{equation}
\Phi: x \mapsto \left(\Gamma(x,s;y_{i_0}), \Gamma(x,s;y_{i_1}), \Gamma(x,s;y_{i_2}), \dots, \Gamma(x,s;y_{i_n}) \right).
\end{equation}
Since for every $x \in B_R$,
\[
|\Phi(x) - \Phi_E(x)| \leq C(n)(\sigma + \delta),
\]
choosing $\sigma < \sigma_2(n)$, $\delta < \delta_2(n)$, such that $C(n)(\sigma + \delta) < d_1(n)/4$, we guarantee that if for some $x \in B_R$, $y \in B_\rho$, it holds that $\Phi(x) = \Phi(y)$, then $x = y$. By the choice of $R_1$, (\ref{eq:ChoiceR1}), and the maximum principle, if $G(\tilde{q}) = G(q)$ for some $q \in B_{\rho r}(p)$ then $\tilde{q} = q$. Since $p$ was an arbitrary point on $M$, it follows that $G$ is an embedding. 

If the heat kernel is replaced by the truncated heat kernel $K_N$, by (\ref{eq:CloseToUnitVec}), the same argument still works to prove that $G^N$ is an embedding. Indeed, define
\[
\Gamma^N(x,t;y) := r^n K_N\left( u^{-1}(x r), s r^2; u^{-1}(y r) \right).
\]
By Lemma \ref{le:TruncHeatKernClose}, for $N$ large enough, the analogue of (\ref{eq:CloseToUnitVec}) still holds,
\[
\left| \nabla \Gamma^N(x,s;y_{i_j}) - \frac{1}{2s} e_j \Gamma_E(0,s;e_1)\right| \leq C(n) (\delta + \rho) + \sigma.
\]

\noindent \textbf{Step 3: Control the dilatation.}

\noindent Let $v \in T_p(M)$, $|v| = 1$. 
First note that if we fix $s=1/2$, the function $(x,y) \mapsto |\nabla \Gamma_E(x,s;y)|$ is maximized if $|x-y|=1$ and its maximal value is
\[
\left|\nabla \Gamma_E\left(x,\tfrac{1}{2};y\right)\right| = (2\pi)^{-\frac{n}{2}} e^{-\frac{1}{2}}.
\]
By (\ref{eq:GammaCloseGammaE}) and by the definition of $R_2$, it follows with (\ref{eq:GradEstimate}) that for all $q \in M$,
\[
\left|\nabla \tilde{K} \left(x,\tfrac{1}{2};q\right)\right| \leq (2\pi)^{-\frac{n}{2}} e^{-\frac{1}{2}} (1+ \sigma).
\]
Consequently, for all $q \in M$,
\[
|\nabla K(p,t;q) | \leq Q \frac{(2\pi)^{-\frac{n}{2}} e^{-\frac{1}{2}}(1 + \sigma)}{t^{(n+1)/2}}.
\]

Now we express $v = v^j \partial_{u^j}$ and write $\hat{v} = (v^1, v^2, \dots, v^n)$.
Because the points $q_i$ form a $(\delta r)$-net, there is an index $k_0$ such that $| \hat{v} - y_{k_0} | < 2 \delta$. 
It follows that
\[
\begin{split}
\hat{v} \cdot \nabla \Gamma\left(0,\tfrac{1}{2};y_{k_0}\right) 
&\geq \hat{v} \cdot \nabla \Gamma_E\left(0,\tfrac{1}{2};y_{k_0}\right) - \sigma |\hat{v}| \\
& \geq |\hat{v}|\, (2\pi)^{-\frac{n}{2}} e^{-\frac{1}{2}}(1-C(n) (\sigma + \delta)).
\end{split}
\]
Therefore,
\[
v \cdot \nabla K\left(p,\frac{r^2}{2};q_{k_0}\right) \geq \frac{1}{Q} \frac{(2\pi)^{-\frac{n}{2}} e^{-\frac{1}{2}}(1-C(n)(\sigma+\delta))}{r^{n+1}}.
\]

As $t = r^2/2$, with $1<Q< Q_1(n,\epsilon)$, $\delta < \delta_3(n,\epsilon)$, $\sigma < \sigma_3(n,\epsilon)$, we find that the map $G$,
\[
G(p) = (2t)^{\frac{n+1}{2}}  ( K(p, t; q_1), \dots, K(p,t; q_{N_0}) ),
\]
is actually an embedding in $\R^{N_0}$ endowed with the maximum norm, such that for every $p \in M$,
\[
1 - \epsilon < (2\pi)^{\frac{n}{2}} e^{\frac{1}{2}} |(dG)_p| < 1 + \epsilon.
\]
Clearly, truncating the heat kernel at large enough $N\geq N_E(n,\kappa,\iota,V,\epsilon,t)$ yields by Lemma \ref{le:TruncHeatKernClose}
\[
1 - \epsilon < (2\pi)^{\frac{n}{2}} e^{\frac{1}{2}} |(dG^N)_p| < 1 + \epsilon.
\]
\end{proof}

It is illustrative to consider a limit case of the map $G$. That is, let $q_i$, $i=1,2,\dots$ be dense in $M$. Then the map $\tilde{G}:M\to\ell^\infty$ defined by $(\tilde{G}(p))_i=(2\pi)^{\frac{n}{2}}e^{\frac{1}{2}} (2t)^{(n+1)/2} K(p,t;q_i)$ is an embedding for $t$ small enough, and as $t \downarrow 0$, the map becomes closer to an isometry. The map $\tilde{G}$ is similar to the Kuratowski embedding $I$, given by $I(p)_i = d(p,q_i)$, which is an exact isometry.

\subsection{Embedding in Euclidean space}
\label{suse:euclidean}

In the previous section, we observed that whenever points $q_i$ form a dense enough net, the map $G$ consisting of heat kernels from the points $q_i$ is an embedding for some $t$ that is almost an isometry if the image space is endowed with the maximum norm. 

Even for very dense nets, the map $G$ may not give rise to an almost isometry when we map into Euclidean space. The next theorem shows that we do get almost an isometry when we weigh the heat kernels from the different points differently.

\begin{Theorem}
\label{Th:EmbDiffWeight}
Let $\epsilon > 0$. Then, there exists a $t_0 = t_0(n,\kappa,\iota,\epsilon)$ such that for all $0<t<t_0$ there exists an $\delta = \delta(n,\kappa,\iota,\epsilon,t)$ such that for all $(M,g) \in \mathcal{M}(n,\kappa,\iota,V)$ and $\delta$-net $q_1, \dots, q_{N_0}$ on $M$, the map
\begin{equation}
H(p) := (2t)^{\frac{n+2}{4}} \frac{1}{V_e} \left(|A_1|^{1/2} K(p,t;q_1), \dots, |A_{N_0}|^{1/2} K(p,t;q_{N_0}) \right),
\end{equation}
is an embedding of $M$ into $\R^{N_0}$ satisfying for all $p \in M$,
\begin{equation}
1 - \epsilon <  |(d H)_p| < 1 + \epsilon.
\end{equation}
Here, $\{A_i\}_{i=1}^{N_0}$ is a partition of $M$ such that for all $i$, $A_i \subset B_{\delta}(q_i)$, and 
\begin{equation}
V_e := \left( \int_{\R^n} (\partial_{x_1} \Gamma_E(0,1/2;y) )^2 dy \right)^{1/2}
= \frac{1}{\sqrt{2} (4\pi)^\frac{n}{4}}.
\end{equation}
Moreover, there exists an $N_E= N_E(n,\kappa,\iota,V,\epsilon,t)$ such that whenever $N\geq N_E$, the same statements hold for the map $H^N$, which is the map $H$ except with the heat kernel $K$ replaced by the truncated version $K_N$,
\begin{equation}
H^N(p) := (2t)^{\frac{n+2}{4}} \frac{1}{V_e} \left(|A_1|^{1/2} K_N(p,t;q_1), \dots, |A_{N_0}|^{1/2} K_N(p,t;q_{N_0}) \right).
\end{equation}
\end{Theorem}

The proof of this theorem has great similarities with that of Theorem \ref{Th:EmbHeatKernInf} in the previous section. In Step 2, we calculate $|(dH)_p|$, which now represents the operator norm as a map to Euclidean space. The proof uses the observation that in Euclidean space $\R^n$, the integral
\begin{equation}
\int_{\R^n} \left(v \cdot \nabla \Gamma_E(0,1/2,y)\right)^2 dy
\end{equation}
is independent of the direction of $v$. In the proof we select points and weights such that the integral is approximated well by a (Riemann) sum. For that, again we need that the heat kernel is close to the Euclidean heat kernel on a ball $B_{r \tilde{R}}(p)$, and that the contribution from heat kernels from points outside this ball is small by the exponential decay of the gradient of the heat kernel. We introduce the necessary estimates in Step 1.

\begin{proof} \textbf{Step 1: Determine a scale $r>0$.}\\
\noindent Let $\epsilon > 0$ be given. Let $R_1 = R_1(n,\epsilon)$ be a radius such that for every $\Gamma$ satisfying the gradient decay (\ref{eq:GradDecCoord}), $1/2 \leq s \leq 2$,
\begin{equation}
\label{eq:IntTailDec}
\int_{\R^n \backslash B_{R_1}(0)} \left|\nabla \Gamma(0,s;y) \right|^2 dy < \epsilon,
\end{equation}
and moreover, if $\{B_i\}_i$ is a partition of $\R^n$, such that the diameter of every $B_i$ is less than $1$,
\begin{equation}
\label{eq:IntTailDecPart}
\sum_{i, B_i \nsubseteq B_{R_1}} \sup_{y \in B_i} |\nabla \Gamma(0,s; y)|^2 |B_i| < \epsilon.
\end{equation}

Let $\sigma > 0$, be less than $\sigma_1(n,\epsilon)$ to be determined later. Set $\alpha = 1/2$. 
Pick $1 < Q < \sqrt{2}$ close enough to $1$. 
To be precise, such that $Q-1 < \epsilon$ and 
\begin{equation}
2(Q-1)C(n,\alpha) \leq \sigma,
\end{equation}
where $C(n,\alpha)$ is as in Lemma \ref{le:CloseToEucl}. 
Select also $R_0 = R(n,\alpha, C_d(n), Q)$ as in Lemma \ref{le:CloseToEucl}, where $C_d(n)$ is the constant in (\ref{eq:DecCoord}).

Set $r_h := r_h(n,\kappa,\iota,\alpha = 1/2,Q)$. 

Let moreover $0< r_3=r_3(n,\kappa,\iota,\epsilon) <r_h/2$ be small enough such that for $t< 2 r_3^2$,
\begin{equation}
\label{eq:NoEffRestMan}
(2t)^{\frac{n+2}{2}} \int_{M\backslash B_{r_h/2}(p)} \left| \nabla K(p,t;q)\right|^2 dq < \epsilon.
\end{equation}
At first sight, the value of $r_3$ may seem to depend on $V$, yet by the Bishop-Gromov inequality, the volume of the manifold $M$ grows at most exponentially with the distance, while the heat kernel decays with the exponential of the distance squared.

Now set 
\begin{equation}
r_0 := \min\left(\frac{r_h}{R_0},\frac{r_h}{R_1}, r_3\right).
\end{equation}
Let $r < r_0$ and $t = r^2/2$. Note that the function $\nabla \Gamma_E(0,s; .)$ is uniformly continuous, with modulus of continuity $\omega$, say. 
Given $\epsilon$, let $\delta$ be smaller than $r \omega(\epsilon)/2$ and let $\{q_i\}_{i = 1}^{N_0}$, be a $\delta$-net.
Now partition the manifold $M$ into $N_0$ sets $A_i \subset B_\delta(q_i)$. 

\noindent \textbf{Step 2: Control the dilatation}

\noindent Let $p \in M$ and let $u:B_{r_h}(p) \to \R^n$ be harmonic coordinates with $u(p) = 0$, satisfying (\ref{eq:BoundsOng}). Define
\[
\tilde{K}(x,s; q) := r^n K( u^{-1} (x r),  s r^2 ; q),
\]
and
\[
\Gamma(x,s;y) := r^n K\left(u^{-1} (x r), s r^2; u^{-1}(y r) \right).
\]
Therefore, if $ v \in T_p(M)$, $|v| = 1$, $v = v^j \partial_{u^j}$, and $\hat{v}=(v^1,\dots,v^n)$,
\begin{equation}
\label{eq:CalcdH}
\begin{split}
|(dH)_p (v) |^2 &= \frac{1}{V_e^2} r^{n+2} 
\sum_{i = 1}^{N_0} (v \cdot \nabla {K}(p,t;q_i))^2 |A_i| \\
&= \frac{1}{V_e^2} \frac{r^{2n+2}}{\left(r^{n+1}\right)^2} 
\sum_{i=1}^{N_0} (\hat{v} \cdot \nabla \tilde{K}(0,\tfrac{1}{2};q_i))^2 \frac{|A_i|}{r^{n}} \\
&= \frac{1}{V_e^2} \sum_{i \in I_{R_1 r}(p)} \left(\hat{v} \cdot \nabla \tilde{K}\left(0,\tfrac{1}{2};q_i\right)\right)^2 \frac{|A_i|}{r^{n}} \\ 
&\qquad  + \frac{1}{V_e^2} \sum_{i \notin I_{R_1 r}(p)} \left(\hat{v} \cdot \nabla \tilde{K}\left(0,\tfrac{1}{2};q_i\right)\right)^2 \frac{|A_i|}{ r^{n}},
\end{split}
\end{equation}
where for each $\tilde{r}>0$ and $\tilde{p}\in M$, $I_{\tilde{r}}(\tilde{p})$ denotes the subset of $\{1, \dots, N_0\}$ such that $A_j \cap B_{\tilde{r}}(p) \neq \emptyset$. We will first show the first term is close to $1$, after which we will prove that the last term is small.

Note that $B_R \subset u(B_{r_h}(p))/r$, where 
\[
R := \frac{2}{\sqrt{Q}}\max(R_0,R_1).
\]
By Lemma \ref{le:CloseToEucl} and our choice of $Q$ and $R_0$, for all $y \in B_R(0)$, $x \in B_2(0)$,
\[
\left| \nabla\Gamma\left(x,\tfrac{1}{2};y\right) 
	- \nabla\Gamma_E\left(x,\tfrac{1}{2};y\right) \right| < \sigma,
\]
which implies that
\begin{multline*}
\left| \sum_{i \in I_{R_1 r}(p)} 
  \left(\hat{v} \cdot \nabla \Gamma\left(0,\tfrac{1}{2};y_i\right) \right)^2 
  	\frac{|A_i|}{r^{n}} 
  - \sum_{i \in I_{R_1 r}} 
  	  \left(\hat{v} \cdot \nabla \Gamma_E\left(0,\tfrac{1}{2};y_i\right)\right)^2
  	  \frac{|A_i|}{r^{n}} \right| \\
< C(n) \sigma \sum_{i \in I_{R_1 r}(p)} \frac{|A_i|}{r^n}
< C(n) \sigma |B_{R_1}(0)|,
\end{multline*}
where $y_i := r^{-1} u(q_i)$ if $q_i \in B_{r_h}(p)$.
By choosing $\sigma < \sigma_1(n,\epsilon)$, we can bound the right-hand side by $\epsilon$. 
Since $\mathrm{diam} \, A_i < \omega(\epsilon)$,
\[
\left| \sum_{i \in I_{R_1 r}(p)} \left(\hat{v} \cdot \nabla \Gamma_E\left(0,\tfrac{1}{2};y_i\right)\right)^2 
\frac{|A_i|}{r^{n}} 
	- \sum_{i \in I_{R_1 r}(p)} \int_{r^{-1} u(A_i)} \left(\hat{v} \cdot \nabla \Gamma_E\left(0,\tfrac{1}{2};y\right)\right)^2 d((r^{-1}u)_\#\Vol)\right| < C(n) \epsilon,
\]
where $((r^{-1}u)_\#\Vol)$ denotes the push-forward under $r^{-1}u$ of the standard volume measure on $M$. Since $u$ satisfies (\ref{eq:BoundsOng}), we find
\begin{multline*}
\left| \sum_{i \in I_{R_1 r}(p)} \int_{r^{-1} u(A_i)} \left(\hat{v} \cdot \nabla \Gamma_E\left(0,\tfrac{1}{2};y\right)\right)^2 d((r^{-1}u)_\#\Vol)
- \sum_{i \in I_{R_1 r}(p)} \int_{r^{-1} u(A_i)} \left(e_1 \cdot \nabla \Gamma_E\left(0,\tfrac{1}{2};y\right)\right)^2 dy \right| \\ < C(n) (Q-1) < C(n) \epsilon.
\end{multline*}

Let us now estimate the last term in (\ref{eq:CalcdH}). By the choice of $R_1$ in (\ref{eq:IntTailDecPart}),
\[
\sum_{i \in I_{r_h/2}(p) \backslash I_{R_1 r}(p) } \left(\hat{v} \cdot \nabla \Gamma\left(x,\tfrac{1}{2};y_i\right) \right)^2 \frac{|A_i|}{r^n} < C(n) \epsilon.
\]
Finally, because of (\ref{eq:NoEffRestMan}),
\[
\sum_{i \in (I_{r_h/2})^c} \left(\hat{v} \cdot \nabla \tilde{K}\left(0,\tfrac{1}{2}; q_i\right)\right)^2 \frac{|A_i|}{r^n} 
= \frac{(r^{n+1})^2 }{r^n} 
	\sum_{i \in (I_{r_h/2})^c} (v \cdot \nabla K(0,t;q_i))^2 |A_i| 
< C(n) \epsilon.
\]

Combining the estimates above, we find
\[
\left| |dH(v)|^2 - 1 \right|  \leq C(n) \epsilon.
\]
It follows immediately from Lemma \ref{le:TruncHeatKernClose} that also
\[
\left| |dH^N(v)|^2 -1 \right| \leq C(n) \epsilon,
\]
for $N$ large enough.
\end{proof}

It follows that particular choices of the points $q_i$ in fact make $G$ close to an isometry: for that we can start with an arbitrary net, apply Theorem \ref{Th:EmbDiffWeight}, and just repeat points $q_i$ in the list with smaller, but uniform weights.

\begin{Theorem}
\label{Th:EmbHeatKernEucl}
Let $\epsilon > 0$. Then, there exists a $t_0 = t_0(n,\kappa,\iota,\epsilon)$ such that for all $0<t \leq t_0$ there exists an $N_0 = N_0(n,\kappa,\iota,V,\epsilon,t)$ such that if $(M,g) \in \mathcal{M}(n,\kappa,\iota,V)$, there exist points $p_1, \dots, p_{N_0}$ on $M$ such that the map $H$ defined by
\begin{equation}
H(p) := (2t)^{\frac{n+2}{4}}\sqrt{2}(4\pi)^{\frac{n}{4}} \lambda (K(p,t;p_1), \dots, K(p,t;p_{N_0})),
\end{equation}
for a certain constant $\lambda > 0$, is an embedding of $M$ into $\R^N$ satisfying for all $p \in M$,
\begin{equation}
1 - \epsilon <  |(d H)_p| < 1 + \epsilon.
\end{equation}
In addition, there exists an $N_E= N_E(n,\kappa,\iota,V,\epsilon,t)$ such that whenever $N\geq N_E$, the same statements hold with every heat kernel $K$ replaced by the truncated version $K_N$.
\end{Theorem}

\begin{proof}
This is an immediate consequence of Theorem \ref{Th:EmbDiffWeight}. This follows because, for any $\lambda>0$, we can replace for each fixed $i$ the point $q_i$ in the formulation of Theorem \ref{Th:EmbDiffWeight} by points $q_i^j$, $j = 1, \dots, N_i$, where $N_i = \lceil|A_i|/\lambda\rceil$, such that
\[
\left| \lambda N_i - |A_i| \right| < \lambda.
\]
After renaming the points $q_i^j$ to $p_k$, it follows that when $\lambda = \lambda(n,\kappa,\iota,V,\epsilon,t)$ small enough,
\[
1 - \epsilon < |(d H)_p| < 1 + \epsilon.
\]
\end{proof}

Again let us consider the continuous version of the map $H$ as well. 
That is, consider the map $\mathcal{H}: M \to L^2(M)$ given by 
\begin{equation}
\mathcal{H}(p)(q) = (2t)^{\frac{n+2}{4}}\frac{1}{V_e} K(p,t;q).
\end{equation}
Then, 
\begin{equation}
|d\mathcal{H}_p(v)|^2 = 
	(2t)^{\frac{n+2}{2}} \frac{1}{V_e^2} \int_M \left( v \cdot \nabla K(p,t,q) \right)^2 dq.
\end{equation}
As $t \to 0$, the right-hand side will converge to $1$. 
In fact, we have the following theorem.

\begin{Theorem}
\label{Th:EmbCont}
Let $\epsilon > 0$. Then there exists a $t_0 = t_0(n,\kappa,\iota,\epsilon)$ such that for all $0< t < t_0$, and all $(M,g) \in \mathcal{M}(n,\kappa,\iota,V)$, the map
\begin{equation}
\mathcal{H}(p)(q) := (2t)^{\frac{n+2}{4}} \sqrt{2}(4\pi)^{\frac{n}{4}} K(p,t,q),
\end{equation}
is an embedding of $M$ into $L^2(M)$ such that
\begin{equation}
1 - \epsilon < |(d\mathcal{H})_p| < 1 + \epsilon.
\end{equation}
\end{Theorem}

The proof of Theorem \ref{Th:EmbCont} is very similar to, but easier than that of Theorem \ref{Th:EmbDiffWeight}. 

\section{Embedding with eigenfunctions}
\label{se:EmbEig}

As was also observed by B\'{e}rard et al. \cite{berard_embedding_1994}, the map $\mathcal{H}$ in Theorem \ref{Th:EmbCont} can be composed with the isometry between $L^2(M)$ and $\ell^2$ given by the eigenfunction expansion. By the estimates on the eigenfunctions in Lemma \ref{le:EigFuncBounds}, we can truncate this map. This yields our main result about embedding with eigenfunctions.

\begin{Theorem}
\label{Th:EmbEigFunc}
Let $\epsilon > 0$. Then there exists a $t_0 = t_0(n,\kappa,\iota,\epsilon)$ such that for all $0 < t < t_0$, there is an $N_E = N_E(n,\kappa,\iota,V,\epsilon,t)$ such that if $N \geq N_E$, for all $(M,g) \in \mathcal{M}(n,\kappa,\iota,V)$, the map
\begin{equation}
\mathcal{F}_N(p) := (2t)^{\frac{n+2}{4}} \sqrt{2}(4\pi)^{n/4} 
\left(e^{-\lambda_1 t} \phi_1(p) , \dots, e^{- \lambda_N t} \phi_N(p) \right)
\end{equation}
is an embedding of $M$ into $\R^N$ such that 
\begin{equation}
1 - \epsilon < |(d\mathcal{F}_N)_p| < 1 + \epsilon.
\end{equation}
\end{Theorem}

\begin{proof}
It immediately follows from Theorem \ref{Th:EmbHeatKernInf} that the map $\mathcal{F}_N$ is an embedding for $N$ larger than some $N_E(n,\kappa,\iota,V,\epsilon,t)$. Indeed, as the truncated map $G^N$ is an embedding, $\mathcal{F}_N$ is one as well.

Next, we note that $L^2(M)$ is isometric to $\ell^2$, by the isometry
\[
U(f)_k = \int_M f(q) \phi_k(q) dq.
\]
If we let $t_0$ be the constant from Theorem \ref{Th:EmbCont}, we know that for $t < t_0$,
\[
\begin{split}
(1 - \epsilon)^2
&< 2 (2t)^{\frac{n+2}{2}} (4\pi)^{n/2} \sum_{k=1}^\infty e^{-2 \lambda_k t} 
	\left( v \cdot \nabla \phi_k(p)\right)^2\\
&= |(d\mathcal{H})_p(v)|^2 \\
&= | (d (U \circ \mathcal{H}) )_p (v)|^2\\
&< (1+\epsilon)^2.
\end{split}
\]

By Lemma \ref{le:EigFuncBounds}, an argument along the lines of the proof of Theorem \ref{le:TruncHeatKernClose} yields that there exists an $N_E= N_E(n,\kappa,\iota,V,\epsilon,t)$ such that the tail can be ignored for $N \geq N_E$.

\end{proof}

\section{Schauder theory for fundamental solutions on charts}
\label{se:SchauderTheory}

In this section, we will show how parabolic Schauder estimates in $\R^n$ yield closeness of fundamental solutions to Euclidean fundamental solutions. 

For $x \in \R^n$ and $r,\sigma>0$, we define the parabolic cylinder 
\begin{equation}
P_{r,\sigma}(x) = \{(y,s)\in \R^n\times\R^+ \, | \, |y - x| < r, 0 < t < \sigma \}.
\end{equation}
We will consider the parabolic operator $L$ given by
\begin{equation}
L u := u_t - a^{ij}(x) \partial_{x^i}\partial_{x^j} u,
\end{equation}
acting on functions $u$ defined on a domain $\Omega \subset \R^n \times \R^+$ such that $P_{R,T}(0)\subset \Omega$ where $R>0$ and $T>0$ are constants. Moreover, we assume that the coefficients $a^{ij}$ form a symmetric matrix $(a^{ij}=a^{ji})$ and satisfy
\begin{subequations}
\label{eq:BoundsOna}
\begin{align}
Q^{-1} (\delta_{ij}) \leq (a_{ij}) &\leq Q (\delta_{ij}) \qquad \text{ as bilinear forms, }\\
\| a_{ij} \|_{C^\alpha} &\leq Q - 1,
\end{align}
\end{subequations}
for some $Q > 1$ and $0 < \alpha <1$. Throughout this section we will additionally assume that $Q < \sqrt{2}$.

By a fundamental solution to $L$ we mean a function $\Gamma$ satisfying $L \Gamma = 0$ for $t>0$, and moreover for every continuous function $f$ on $\overline{\Omega}$,
\begin{equation}
\lim_{t \downarrow 0} \int_\Omega \Gamma(x,t; y) f(x) dx = f(y).
\end{equation}

The function $Z$ is defined as the fundamental solution of the equation with the coefficients frozen at $y$, that is
\begin{equation}
Z(x,t;y) 
:= \frac{\sqrt{\det(a_{ij}(y))}}{(2 \sqrt{\pi})^n  t^{n/2}}
	\exp\left[ - \frac{\sum_{i,j=1}^n a_{ij}(y) (x^i - y^i)(x^j - y^j)}{4t}\right].
\end{equation}

The following lemma is the main technical result from the Schauder estimates that we will need. It states that every fundamental solution to $L$ that decays sufficiently fast, is in the space-coordinate $C^1$-close to $Z$. The main part of the proof is based on Lemmas \ref{le:FundSolCloseToEucl} and \ref{le:DiffFundSolSmall}.

\begin{Lemma}
\label{le:CloseToEucl}
Let $\Gamma$ be a fundamental solution of $L$ satisfying
\begin{equation}
\Gamma(x,t;y) \leq \frac{C_d}{t^{n/2}} \exp\left[ - \frac{|x-y|^2}{8t} \right],
\end{equation}
for a constant $C_d$ and for all $x, y \in B_{R}(0)$, $0<t\leq 2$. Then for $R\geq R(n,C_d,\alpha,Q)$, for all $(x,t) \in P_{R-1,4}(0) \backslash P_{\frac{1}{2},\frac{1}{4}}(y)$, if $|x| \leq 2$ or $|y| \leq 2$,
\begin{subequations}
\label{eq:FundCloseToFrozen}
\begin{align}
\left| \Gamma(x, t ; y) - Z(x,t; y) \right| 
	&\leq (Q - 1) C(n, \alpha),\\
\left| \nabla \Gamma(x, t ; y)- \nabla Z(x,t; y) \right|
	&\leq (Q - 1) C(n, \alpha),
\end{align}
\end{subequations}
and
\begin{subequations}
\label{eq:FundCloseToEucl1}
\begin{align}
\left| \Gamma(x, t ; y) - \Gamma_E(x,t; y) \right| 
	&\leq (Q - 1) C(n, \alpha),\\
\left| \nabla \Gamma(x, t ; y)- \nabla \Gamma_E(x,t; y) \right|
	&\leq (Q - 1) C(n, \alpha).
\end{align}
\end{subequations}

\end{Lemma}

\begin{proof}
The inequalities (\ref{eq:FundCloseToFrozen}) follow immediately from Lemmas \ref{le:FundSolCloseToEucl} and \ref{le:DiffFundSolSmall} below.

To show the inequalities in (\ref{eq:FundCloseToEucl1}), we can calculate explicitly
\[
\partial_{x^j} Z(x,t;y) = \frac{1}{2t}a_{ij}(y)(x^i-y^i) Z(x,t;y),
\]
and
\[
\nabla \Gamma_E(x,t; y) = \frac{1}{2t} (x-y) \Gamma_E(x,t;y).
\]
It follows that for $(x,t) \notin P_{\frac{1}{2},\frac{1}{4}}(y)$,
\begin{align*}
\left| \Gamma_E(x, t ; y) - Z(x,t; y) \right| 
	&\leq (Q - 1) C(n, \alpha),\\
\left| \nabla \Gamma_E(x, t ; y)- \nabla Z(x,t; y) \right|
	&\leq (Q - 1) C(n, \alpha),
\end{align*}
and therefore also if $|y| \leq 2$ or $|x| \leq 2$,
\begin{align*}
\left| \Gamma(x, t ; y) - \Gamma_E(x,t; y) \right| 
	&\leq (Q - 1) C(n, \alpha),\\
\left| \nabla \Gamma(x, t ; y)- \nabla \Gamma_E(x,t; y) \right|
	&\leq (Q - 1) C(n, \alpha).
\end{align*}
\end{proof}

We first use the parametrix method to construct a fundamental solution that is close to $Z$. In the context of heat kernels on manifolds, the parametrix method usually has a related but slightly different meaning, and is at the basis of the Minakshisundaram-Pleijel expansion. Both methods start with an initial guess for the fundamental solution. However, our initial guess is courser, and therefore we need to use less regularity properties. We will follow the method as presented in Friedman's book \cite{friedman_partial_2008}. The precise statement is contained in the following lemma.

\begin{Lemma}
\label{le:FundSolCloseToEucl}
There is a fundamental solution $\Gamma$ for $L$ on the domain $\Omega$ such that for $y \in B_R$, and $(x,t) \in P_{R,T}(0)$,
\begin{subequations}
\label{eq:FundCloseToEucl}
\begin{align}
\left| \Gamma(x, t ; y) - Z(x,t; y) \right| 
	&\leq \frac{(Q - 1) C(n, \alpha)}{t^{(n-\alpha)/2}}\exp\left[ - \frac{|x- y|^2}{8t} \right],\\
\left| \nabla \Gamma(x, t ; y)- \nabla Z(x,t; y) \right|
	&\leq \frac{(Q - 1) C(n, \alpha)}{t^{(n+1-\alpha)/2}}\exp\left[ - \frac{|x- y|^2}{8t} \right].
\end{align}
\end{subequations}
In particular, there is a $C_d = C_d(n,\alpha)$ such that for $t \leq 4$, $\Gamma$ satisfies
\begin{equation}
\Gamma(x,t;y) \leq \frac{C_d}{t^{n/2}} \exp\left[ - \frac{|x-y|^2}{8t} \right].
\end{equation}
\end{Lemma}

\begin{proof}
We denote the fundamental solution we are looking for by $\Gamma$. Following Friedman \cite{friedman_partial_2008}, we look for a fundamental solution $\Gamma$ of the form
\[
\Gamma(x,t; y) = Z(x,t;y) 
+ \int_0^t \int_{B_R} Z(x,t-\sigma;\eta) \Phi( \eta, \sigma; y ) d\eta\, d\sigma,
\]
for some function $\Phi$. Since 
\[
\Phi(x,t;y) = LZ(x,t;y) + \int_0^t \int_{B_R} LZ(x,t-\sigma;\eta) \Phi(\eta,\sigma;y) d\eta \, d\sigma,
\]
we can solve for the function $\Phi$ by iteration,
\[
\Phi(x,t; y) = \sum_{i = 1}^\infty (LZ)_i(x,t;y),
\]
where $(LZ)_1 := LZ$ and for $i \in \mathbb{N}$,
\[
(LZ)_{i+1}(x,t;y) := \int_0^t \int_{B_R} LZ(x,t-\sigma;\eta) (LZ)_i(\eta,\sigma; y) d\eta \, d\sigma.
\]
Inspecting the proof of \cite[Ch. 1, (4.15)]{friedman_partial_2008}, we find that the following estimates hold for $\Phi$. For every $0 < \lambda < 1/Q$, there is a $C = C(n,\alpha,\lambda)$ such that
\[
\left| \Phi(x,t; y) \right| \leq
	\frac{(Q-1) C}{t^{(n+2-\alpha)/2}}
		\exp\left[ - \frac{\lambda |x-y|^2}{4t} \right].
\]
Let us select $\lambda = 1/\sqrt{2}$. 
With \cite[Ch. 1, Lemma 3]{friedman_partial_2008}, it follows that
\[
\int_0^4 \int_{B_R} Z(x,t-\sigma;\eta) |\Phi(\eta,\sigma;y)| d\eta d\sigma \leq 
\frac{(Q-1) C(n,\alpha)}{t^{(n-\alpha)/2}} 
	\exp\left[ - \frac{| x - y|^2}{8t} \right]
\]
As a consequence,
\[
|\Gamma(x,t;y)| \leq \frac{C_d(n,\alpha)}{ t^{n/2}} 
\exp\left[- \frac{|x-y|^2}{8t}\right].
\]
Note also that
\[
\nabla\Gamma(x,t;y) = \nabla Z(x,t;y) 
	+ \int_0^t \int_{B_R} \nabla Z(x,t-\sigma; \eta) 
		\Phi(\eta,\sigma; y) \, d\eta \, d\sigma.
\]
Again using \cite[Ch. 1, Lemma 3]{friedman_partial_2008},
\[
\int_0^4 \int_{B_R} |\nabla Z(x,t-\sigma;\eta)| |\Phi(\eta,\sigma;y)| \, d\eta \, d\sigma \leq 
\frac{(Q-1) C(n,\alpha)}{t^{(n+1-\alpha)/2}} 
	\exp\left[ - \frac{| x - y|^2}{8t} \right]
\]
\end{proof}

By the previous lemma, we know that there is at least one heat kernel close to the Euclidean heat kernel. 
The following lemma shows that any other fundamental solution with sufficient decay is close to this heat kernel by parabolic Schauder estimates.

\begin{Lemma}
\label{le:DiffFundSolSmall}
Let $\epsilon > 0$, $C_d>0$, and $0 < \alpha < 1$ be given. Then there exists an $R=R(n,\alpha,C_d,\epsilon)$ as follows. Suppose $\Gamma_1$ and $\Gamma_2$ are two fundamental solutions for the operator $L$ on $P_{R,4}(0)$, satisfying for $i=1,2$, and $0 < t \leq 4$,
\begin{equation}
\label{eq:ExpDecGammai}
\Gamma_i(x,t;y) \leq \frac{C_d}{t^{n/2}} \exp\left[ - \frac{|x-y|^2}{8t} \right].
\end{equation}
Then 
\begin{equation}
\label{eq:EstXSmall}
\| \Gamma_1(.,.;y) - \Gamma_2(.,.;y) \|_{C^{2,\alpha}(P_{2,4})} \leq \epsilon,
\end{equation}
for all $y \in B_R(0)$, and
\begin{equation}
\label{eq:EstXiSmall}
\left\| \Gamma_1(.,.;y) - \Gamma_2(.,.;y) \right\|_{C^{2,\alpha}(P_{R-1,4})} \leq \epsilon,
\end{equation}
for all $y \in B_2(0)$.

\end{Lemma}

\begin{proof}
Consider $\Gamma_1 - \Gamma_2$. By interior parabolic estimates \cite[Ch. 4, Theorem 4]{friedman_partial_2008} and the maximum principle, there is a $C_1 = C_1(n,\alpha)$ such that
\[
\left\|\Gamma_1(.,.;y) - \Gamma_2(.,.;y)\right\|_{C^{2,\alpha}(P_{2,4}(0))}
\leq C_1(n,\alpha) \left\|\Gamma_1(.,.;y) - \Gamma_2(.,.;y )\right\|_{C^0(\partial B_{3}(0) \times [0,4])}.
\]
Because of the decay (\ref{eq:ExpDecGammai}), for $|y|\geq R_1 = R_1(n,\alpha,C_d,\epsilon)$,
\[
\left\|\Gamma_1(.,.;y) - \Gamma_2(.,.;y )\right\|_{C^0(\partial B_{3}(0) \times [0,4])} < \frac{\epsilon}{C_1(n,\alpha)},
\]
so that (\ref{eq:EstXSmall}) holds for $|y| \geq R_1$.

Similarly, there exists a constant $C_2(n,\alpha)$ such that for all $x \in B_{R-1}(0)$,
\[
\left\|\Gamma_1(.,.;y) - \Gamma_2(.,.;y)\right\|_{C^{2,\alpha}(P_{1/2,4}(x))}
\leq C_2(n,\alpha) \left\|\Gamma_1(.,.;y) - \Gamma_2(.,.;y )\right\|_{C^0(P_{1,4}(x))}.
\]
It follows by this bound and the maximum principle that there is a constant $C_3 = C_3(n,\alpha)$, such that for all $R \geq 2$, 
\[
\left\|\Gamma_1(.,.;y) - \Gamma_2(.,.;y) \right\|_{C^{2,\alpha}(P_{R-1,4}(0))} 
\leq C_3(n,\alpha) \left\|\Gamma_1(.,.;y) - \Gamma_2(.,.;y)\right\|_{C^0(\partial B_{R}(0) \times [0,4])}.
\]
Consequently, if $|y| < R_1$, and $R \geq R_2=R_2(n,\alpha,C_d,\epsilon)$, (\ref{eq:EstXiSmall}) holds, and (\ref{eq:EstXSmall}) also holds for $|y| \leq R_1$.
\end{proof}

\appendix
\section{Quantitative Harmonic Radius Estimates}

In \cite{anderson_compactness_1992} Anderson and Cheeger prove that the $C^\alpha$-harmonic radius can be bounded uniformly from below in the class of $n$-dimensional manifolds with a fixed lower bound for the Ricci curvature and injectivity radius. 
We used this result to bound the number of eigenfunctions or heat kernels needed to embed a manifold in Euclidean space.
However, the proof by Anderson and Cheeger is qualitative, and does not yield an estimate for the harmonic radius. 
Consequently, we do not yet have a quantitative estimate of the number of eigenfunctions or heat kernels needed.

In contrast, under sectional curvature bounds rather than Ricci curvature lower bounds, Jost and Karcher \cite{jost_geometrische_1982} explicitly constructed harmonic coordinates with quantitative estimates. 

The purpose of this appendix is therefore to make a (more) quantitative estimate on the harmonic radius, under the assumptions of the paper by Anderson and Cheeger. 
The ideas we use are similar to ideas used by Cheeger and Colding in \cite{cheeger_lower_1996}, the series of papers by Colding \cite{colding_shape_1996,colding_large_1996,colding_ricci_1997}, and by Cheeger and Colding \cite{cheeger_structure_1997,cheeger_structure_2000,cheeger_structure_2000-1}.
However, where these papers typically argue from a local volume lower bound, we assume a lower bound on the injectivity radius.
In this less general setting, it is possible to show H\"{o}lder continuity of the metric in terms of distance function coordinates.

\subsection{Setup and construction of harmonic coordinates}
\label{se:HRsetup}

Throughout the appendix, $(M,g)$ will be in the class $\mathcal{M}(n,\Lambda,D)$ of smooth Riemannian manifolds with or without boundary such that
\begin{equation}
\Ric \geq - (n-1) \Lambda^2, \quad (\Lambda > 0), \qquad \inj_M \geq \iota.
\end{equation}
For two points $p,q \in M$ that are nonconjugate, we will let $\gamma_{p,q}:[0,\infty)\to M$ denote the (continuation of the) minimizing geodesic from $p$ to $q$ with arclength parametrization.
We will select an arbitrary point $p \in M$ such that $d(p,\partial M) > \iota$.
One of the goals is to construct harmonic coordinates on a ball $B_r(p)$ for some radius $r$ that we can estimate from below, and give bounds on the H\"{o}lder continuity of the metric coefficients.

We will select points $p_1, \dots, p_n$ such that for $i = 1, \dots, n$,
\begin{equation}
\frac{3\iota}{16} \leq d(p_i,p) \leq \frac{\iota}{4}.
\end{equation}
The first goal is to obtain regularity for the Busemann functions $\rho_{p_i}:=d(p_i,.)$. 
The lower bound on the Ricci curvature and injectivity radius immediately imply a supremum bound on the Laplacian of the $\rho_{p_i}$. 
Indeed, Lemma 1.4 of \cite{anderson_compactness_1992} states that if $\Ric \geq -(n-1) \Lambda^2$ ($\Lambda > 0$) as long as $\rho_{p_1} \leq \iota/2$,
\begin{equation}
\label{eq:BoundLaplRho}
|\Delta \rho_{p_1}| \leq (n-1) \Lambda \coth( \Lambda \rho_{p_1} ).
\end{equation}
For completeness, we have repeated the proof of this bound at the end of this section.
We will combine this supremum bound with the Bochner formula to show that the average squared norm of the Hessian of distance functions is small. By the segment inequality we may from this extract H\"{o}lder bounds on $g(\nabla \rho_{p_i}, \nabla \rho_{p_j})$ by a Morrey / Campanato argument.

Subsequently, we choose the points $p_i$ in such a way that at the origin, the gradients $\nabla \rho_{p_i}$ form an orthonormal system.
Afterwards, we construct harmonic coordinates $b_i$ by solving the following Dirichlet problems
\begin{equation}
\begin{cases}
\Delta b_i = 0, & \text{ on } B_r(p),\\
b_i  = \rho_i, & \text{ on } \partial B_r(p).
\end{cases}
\end{equation}

Because the supremum norm estimate (\ref{eq:BoundLaplRho}) on the Laplacian of Busemann functions is central to the estimates that follow below, we will now include its proof for completeness.

\begin{Lemma}[{Anderson and Cheeger \cite{anderson_compactness_1992}}]
Let $(M,g) \in \mathcal{M}(n,\Lambda,\iota)$ be a smooth Riemannian manifold. Let $p_1 \in M$.
Moreover, assume either
\begin{itemize}
\item $\rho_{p_1} < \iota / 2$ and $d(p_1,\partial M) \geq \iota$, or 
\item for some $p \in M$, $\rho_p \leq \iota/4$, $d(p,\partial M) \geq \iota$, and $d(p_1,p) \leq \iota /4$.
\end{itemize}
Then
\begin{equation}
|\Delta \rho_{p_1}| \leq (n-1) \Lambda \coth( \Lambda \rho_{p_1} ).
\end{equation}
\end{Lemma}

\begin{proof}
We only prove the statement under the first assumption, as the proof under the other assumption is completely similar. By Laplacian comparison, on $B_\iota(p_1)$, 
\[
\Delta \rho_{p_1} \leq (n-1) \Lambda \coth(\Lambda \rho_{p_1}).
\]
Let $x \in B_{\iota/2}(p_1)$, denote $d = d(p_1,x)$, and consider $y = \gamma_{p_1,x}(2d)$.
Then also
\[
\Delta \rho_y \leq (n-1) \Lambda \coth(\Lambda \rho_y).
\]
By the triangle inequality, $\rho_{p_1} + \rho_{y}$ is minimal on $\gamma_{p_1,y}$. 
Hence, at the point $x$,
\[
0 \leq \Delta \rho_{p_1} + \Delta \rho_y \leq \Delta \rho_{p_1} + (n-1) \Lambda \coth(\Lambda \rho_y),
\]
Since $\rho_y(x) \geq \rho_{p_1}(x)$,
\[
-(n-1) \Lambda \coth(\Lambda \rho_{p_1}(x)) \leq \Delta \rho_{p_1}(x),
\]
from which the bound follows.
\end{proof}

\subsection{Small average Hessian of distance functions}
\label{se:SmallHessian}

Let $(M,g)\in \mathcal{M}(n,\Lambda,\iota)$ and $p \in M$ be as in Section \ref{se:HRsetup}.
Let $r < \iota / 8$ and let $p_1 \in M$ be such that $\iota/8 < d(p_1,p) < \iota/4$.
Write $\rho = \rho_{p_1} = d(p_1, . )$.

By the Bochner-Weitzenbock formula
\begin{equation}
0 = \Delta |\nabla \rho|^2 = |\Hess_\rho|^2 + ( \nabla \Delta \rho, \nabla \rho ) + \Ric(\nabla \rho, \nabla \rho).
\end{equation}
As in \cite{anderson_compactness_1992}, after integration of this formula and integration by parts,
\begin{equation}
\int_{B_r}|\Hess_\rho|^2 \leq (n-1) \Lambda^2 \Vol(B_r) + \int_{B_r} (\Delta \rho)^2 + \int_{\partial B_r} |\Delta \rho|.
\end{equation}
We multiply both sides of the inequality by $r^2/\Vol(B_r(p))$ and use the estimate (\ref{eq:BoundLaplRho}) to obtain
\begin{equation}
\begin{split} 
\frac{r^{2}}{\Vol(B_r(p))} & \int_{B_r(p)} |\Hess_\rho|^2  \\ & \leq (n-1) (\Lambda r)^{2} +  (n-1)^2 (\Lambda r)^2 \| \coth (\Lambda \rho) \|^2_{L^\infty(B_r(p))}\\
&\quad + r \frac{\Vol(\partial B_r(p))}{\Vol(B_r(p))} (n-1) (\Lambda r) \| \coth( \Lambda \rho)\|_{L^\infty(B_r(p))}.
\end{split}
\end{equation}
By Bishop-Gromov volume comparison, we know that
\begin{equation}
\begin{split}
\frac{\Vol(B_r(p))}{\Vol(\partial B_r(p))} =  \frac{\int_0^r \Vol(\partial B_s(p)) ds}{\Vol(\partial B_r(p))} \geq \frac{\int_0^r \Vol_\Lambda(\partial B_s) ds}{\Vol_\Lambda(\partial B_r)} = \frac{\Vol_\Lambda(B_r)}{\Vol_\Lambda(\partial B_r)},
\end{split}
\end{equation}
where $\Vol_\Lambda(\partial B_r)$ and $\Vol_\Lambda(B_r)$ denote respectively the volume of the boundary of a ball of radius $r$, and the volume of a ball of radius $r$, in the simply connected model case of constant sectional curvature $-\Lambda^2$.
Hence, 
\begin{equation}
\label{eq:HessSquared}
\begin{split}
\frac{r^{2}}{\Vol(B_r(p))} & \int_{B_r(p)} |\Hess_\rho|^2 
\\ & \leq 
(n-1) (\Lambda r)^{2} 
+  (n-1)^2 (\Lambda r)^2 \|\coth (\Lambda \rho)\|^2_{L^\infty(B_r(p))} \\
& \quad + r \frac{\Vol_\Lambda(\partial B_r)}{\Vol_\Lambda (B_r)} (n-1) (\Lambda r) \|\coth( \Lambda \rho)\|_{L^\infty(B_r(p))} \\
&=: (\Lambda r) F(n, \Lambda r, \| \coth(\Lambda \rho) \|_{L^\infty(B_r(p))}),
\end{split}
\end{equation}
where the last line serves to define $F$, which is nondecreasing in its second and third arguments. 
As the left-hand side is dimensionless, we expect that this will imply certain regularity.

To obtain a (crude) estimate on $F$, recall that
\begin{equation}
\Vol_\Lambda(\partial B_r) = \Omega_n \frac{\sinh^{n-1}(\Lambda r)}{\Lambda^{n-1}},
\end{equation}
where $\Omega_n$ is the total solid angle in the simply connected hyperbolic model space of constant sectional curvature $-1$. 
Since $\sinh^{n-1}(.)$ is convex, we may apply Jensen's inequality to conclude
\begin{equation}
\begin{split}
\Vol_\Lambda(B_r) &= \Omega_n \int_0^r \frac{\sinh^{n-1}(\Lambda s)}{\Lambda^{n-1}}ds \\
&= \frac{\Omega_nr}{\Lambda^{n-1}}\frac{1}{ \Lambda r} \int_0^{\Lambda r} \sinh^{n-1} u \, du \\
&\geq \frac{\Omega_n r}{\Lambda^{n-1}} \sinh^{n-1}\left( \frac{1}{\Lambda r} \int_0^{\Lambda r} u \, du \right)\\
&\geq \frac{\Omega_n r}{\Lambda^{n-1}} \sinh^{n-1}(\Lambda r/2).
\end{split}
\end{equation}
Therefore,
\begin{equation}
r \frac{\Vol_\Lambda(\partial B_r)}{\Vol_\Lambda (B_r)} \leq \frac{\sinh^{n-1}(\Lambda r)}{\sinh^{n-1}(\Lambda r / 2)} = 2^{n-1}\cosh^{n-1}(\Lambda r / 2),
\end{equation}
and the following estimate holds for $F$:
\begin{equation}
\begin{split}
F(n&, \Lambda r, \| \coth(\Lambda \rho) \|_{L^\infty(B_r(p))}) \\ 
& \leq (n-1) \Lambda r +  (n-1)^2 \Lambda r \|\coth^2 (\Lambda \rho) \|_{L^\infty(B_r(p))}\\
& \quad +  2^{n-1} (n-1) \cosh^{n-1}(\Lambda r / 2) \| \coth( \Lambda \rho) \|_{L^\infty(B_r(p))}.
\end{split}
\end{equation}

\subsection{Segment inequality}

The inequality (\ref{eq:HessSquared}) will imply H\"{o}lder regularity of the inner products of gradients of distance functions by a Morrey / Campanato argument that we will present in the next section.
A technical ingredient that we will use is the segment inequality as introduced by Cheeger and Colding \cite{cheeger_lower_1996}.

Let $A_1$ and $A_2$ be two Borel sets on a Riemannian manifold $(M,g)$, which satisfies 
\begin{equation}
\Ric \geq -(n-1) \Lambda^2.
\end{equation}
Let $W \subset M$ be an open subset such that for every $y_1 \in A_1$ and $y_2 \in A_2$ the minimal geodesic $\gamma_{y_1,y_2}$ (with arclength parametrization) is contained in $W$.
Let $h$ be a nonnegative integrable function on $W$.
Also assume for simplicity that $W$ does not contain a pair of conjugate points.
The segment inequality states that
\begin{equation}
\label{eq:segment}
\begin{split}
\int_{A_1 \times A_2} & \int_0^{d(y_1,y_2)} h(\gamma_{y_1,y_2}(s))ds \\ & \leq c(n,\Lambda D)
  D [ \Vol(A_1) +  \Vol(A_2)] \int_W h,
\end{split}
\end{equation}
where
\begin{equation}
D := \sup_{y_1 \in A_1, y_2 \in A_2} d(y_1,y_2),
\end{equation}
and 
\begin{equation}
c(n,\Lambda s) := \sup_{0< s/2 \leq u \leq s} \frac{\Vol_\Lambda(\partial B_s)}{\Vol_\Lambda(\partial B_u)} = 2^{n-1} \cosh^{n-1}\left(\frac{\Lambda s}{2} \right).
\end{equation}
\subsection{H\"{o}lder continuity of inner product gradient of distance functions}

We select points $p_1, p_2 \in M$ such that $3\iota/16 < d(p_i,p) \leq \iota/4$. 
The main result in this section is the H\"{o}lder continuity of the inner product of the gradients of two distance functions $\rho_i:= \rho_{p_i}$, as expressed by the following lemma.

\begin{Lemma}
\label{le:Holder}
Let $(M,g) \in \mathcal{M}(n, \Lambda, \iota)$ and $p\in M$ be as above. 
Let $x_1, x_2 \in B_{\iota/64}(p)$.
Then,
\begin{equation}
|g(\nabla \rho_1, \nabla \rho_2)(x_1) - g(\nabla \rho_1, \nabla \rho_2)(x_2) | \leq C(n, \Lambda d(x_1,x_2), \Lambda \iota) \Lambda^{1/2} d(x_1,x_2)^{1/2},
\end{equation}
where
\begin{equation}
\label{eq:DefC}
C(n, \Lambda r, \Lambda \iota) = 6 \left(12 \frac{\Vol_{\Lambda}(B_{4 r})}{\Vol_\Lambda(B_{r})} c(n, 3 \Lambda r ) F(n, 3 \Lambda r, \coth(\Lambda \iota/16))\right)^{1/2}.
\end{equation}
\end{Lemma}

We will prove the lemma at the end of the section. 
First, define $f: M \to \R$ by $f(x) := g(\nabla \rho_1, \nabla \rho_2)(x)$, and set
\begin{equation}
(f)_{x,r} := \frac{1}{\Vol(B_r(x))} \int_{B_r(x)} f.
\end{equation}

Consider two points $x_i \in B_{\iota/64}(p)$, $(i=1,2)$, and radii $r_i < \iota / 32$. 
Let $D := d(x_1,x_2) + r_1 + r_2$.
Set
\begin{equation}
s_1 := \min\left(d(x_1,x_2),\max\left(0, \frac{d(x_1,x_2)}{2} + \frac{r_2 - r_1}{2} \right) \right),
\end{equation}
and $s_2 := d(x_1,x_2) - s_1$. 
\begin{Lemma}
Let $x_i \in B_{\iota/64}(p)$ and $r_i \leq \iota/32$ for $i=1,2$. Then, with the above notation,
\begin{equation}
\label{eq:CompareAv1}
\begin{split}
\left| (f)_{x_1,r_1} - (f)_{x_2,r_2} \right|^2 
& \leq 2 \left[\frac{\Vol_\Lambda(B_{D + s_1})}{\Vol_\Lambda(B_{r_1})} + \frac{\Vol_\Lambda(B_{D+s_2})}{\Vol_\Lambda(B_{r_2})}\right] c(n,\Lambda D) \\
  & \qquad \times F(n,\Lambda D, \coth(\Lambda \iota/16) ) (\Lambda D)
\end{split}
\end{equation}
\end{Lemma}

\begin{proof}
Note that for a vector field $X$ on $M$,
\[
\begin{split}
\nabla_X g(\nabla \rho_1, \nabla \rho_2) &= g(\nabla_X \nabla \rho_1, \nabla \rho_2) + g(\nabla \rho_1, \nabla_X \nabla \rho_2)\\
&= \Hess_{\rho_1} (X, \nabla \rho_2) + \Hess_{\rho_2}(\nabla \rho_1, X),
\end{split}
\]
so that for a vector $v\in T_p(M)$, and $\|v \|=1$,
\[
| D_v g(\nabla \rho_1, \nabla \rho_2) | \leq |\Hess_{\rho_1}|  + |\Hess_{\rho_2}|.
\]
We write $\gamma_{y_1,y_2}$ for the minimizing geodesic between $y_1$ and $y_2$, with arclength parametrization.
We estimate
\[
\begin{split}
\left| (f)_{x_1,r_1} - (f)_{x_2,r_2} \right|^2 
& = \left| \frac{1}{\Vol(B_{r_1}(x_1))} \int_{B_{r_1}(x_1)} f(y_1) - \frac{1}{\Vol(B_{r_2}(x_2))} \int_{B_{r_2}(x_2)} f(y_2) \right|^2 \\
& = \left| \frac{1}{\Vol(B_{r_1}(x_1))} \frac{1}{\Vol(B_{r_2}(x_2))} \int_{B_{r_1}(x_1)}\int_{B_{r_2}(x_2)} [f(y_1) - f(y_2)] \right|^2\\
& \leq \frac{1}{\Vol(B_{r_1}(x_1))}\frac{1}{\Vol(B_{r_2}(x_2))} \int_{B_{r_1}(x_1)} \int_{B_{r_2}(x_2)} \left| f(y_1) - f(y_2) \right|^2
\end{split}
\]

Now note that
\[
\begin{split}
\left| f(y_1) - f(y_2)\right|^2 
  &\leq \left| \int_{0}^{d(y_1,y_2)} \partial_s (f (\gamma_{y_1,y_2}(s)))  ds \right|^2 \\
  &\leq d(y_1,y_2) \int_0^{d(y_1,y_2)} |\partial_s(f (\gamma_{y_1,y_2}(s)) )|^2 ds \\
  &\leq d(y_1,y_2) \int_0^{d(y_1,y_2)} \left( |\Hess_{\rho_1}| + |\Hess_{\rho_2}|\right)^2\circ \gamma_{y_1,y_2}(s) ds \\
  &\leq 2 d(y_1,y_2) \int_0^{d(y_1,y_2)} \left( |\Hess_{\rho_1}|^2 + |\Hess_{\rho_2}|^2 \right) \circ \gamma_{y_1,y_2}(s) ds.
\end{split}
\]
Since we defined $D = d(x_1,x_2) + r_1 + r_2$, it holds that
\[
D \geq \sup_{y_1 \in A_1, y_2 \in A_2} d(y_1,y_2).
\]
Set $x = \gamma_{x_1,x_2}(s_1)$, that is,
\[
x =
\begin{cases}
 x_1, & \text{if } r_1 \geq r_2 + d(x_1,x_2),\\ 
 x_2, & \text{if } r_2 \geq r_1 + d(x_1,x_2),\\
 \gamma_{x_1,x_2}\left( \frac{-r_1 + d(x_1,x_2) + r_2}{2} \right), & \text{otherwise},
\end{cases}
\]
and
\[
R =
\begin{cases}
 r_1, & \text{if } r_1 \geq r_2 + d(x_1,x_2),\\ 
 r_2, & \text{if } r_2 \geq r_1 + d(x_1,x_2),\\
 D/2, & \text{otherwise}.
\end{cases}
\]
Note that for two points $y_i \in B_{r_i}(x_i)$ it holds that $d(y_i,x) < R$, $(i=1,2)$ and that the geodesic $\gamma_{y_1,y_2}$ is contained in $B_{D}(x)$.
Then, the segment inequality (\ref{eq:segment}), applied with $h = |\Hess_{\rho_1}|^2 + |\Hess_{\rho_2}|^2$, yields
\[
\begin{split}
\left| (f)_{x_1,r_1} - (f)_{x_2,r_2} \right|^2 
& \leq 2 \left[\frac{1}{\Vol(B_{r_1}(x_1))} + \frac{1}{\Vol(B_{r_2}(x_2))}\right] c(n,\Lambda D) \\
  &\qquad \times D^2 \int_{B_{D}(x)} \left( \left| \Hess_{\rho_1}  \right|^2 + \left| \Hess_{\rho_2}\right|^2\right).
\end{split}
\]
For all $y \in B_{D}(x)$, it holds that $d(y,p) \leq \iota/8$. 
Hence, the inequality (\ref{eq:HessSquared}) implies
\[
\begin{split}
\left| (f)_{x_1,r_1} - (f)_{x_2,r_2} \right|^2 
& \leq 2 \left[\frac{\Vol(B_{D}(x))}{\Vol(B_{r_1}(x_1))} + \frac{\Vol(B_{D}(x))}{\Vol(B_{r_2}(x_2))}\right] c(n,\Lambda D) \\
  & \qquad \times F(n,\Lambda D, \coth(\Lambda \iota / 16)) (\Lambda D).
\end{split}
\]
Note that the ball $B_{D}(x)$ is contained in $B_{D+s_i}(x_i)$, for $i=1,2$, so that from the Bishop-Gromov inequality it follows that
\[
\frac{\Vol(B_{D}(x))}{\Vol(B_{r_i}(x_i))} \leq \frac{\Vol(B_{D+s_i}(x_i))}{\Vol(B_{r_i}(x_i))} \leq \frac{\Vol_\Lambda(B_{D+s_i})}{\Vol_\Lambda(B_{r_i})}.
\]
Consequently,
\begin{equation}
\label{eq:CompareAv2}
\begin{split}
\left| (f)_{x_1,r_1} - (f)_{x_2,r_2} \right|^2 
& \leq 2 \left[\frac{\Vol_\Lambda(B_{D+s_1})}{\Vol_\Lambda(B_{r_1})} + \frac{\Vol_\Lambda(B_{D+s_2})}{\Vol_\Lambda(B_{r_2})}\right] c(n,\Lambda D) \\
  & \qquad \times F(n,\Lambda D, \coth(\Lambda \iota/16) ) (\Lambda D)
\end{split}
\end{equation}
\end{proof}

Let us consider two special cases. 
In the first case, $x_1=x_2$, $r_2 = 2 r_1$ and in the second case $r_1 = r_2 = d(x_1,x_2)$, so that in both cases $D = 3r_1$.
Consequently, the following estimate holds for both cases
\begin{equation}
\label{eq:CompareAv3}
\begin{split}
\left| (f)_{x_1,r_1} - (f)_{x_2,r_2} \right|^2 
& \leq 4 \frac{\Vol_\Lambda(B_{4 r_1})}{\Vol_\Lambda(B_{r_1})} c(n,3 \Lambda r_1) \\
  & \qquad \times F(n, 3 \Lambda r_1, \coth(\Lambda \iota/16))  (3 \Lambda r_1).
\end{split}
\end{equation}
With the definition of the constant $C$ in (\ref{eq:DefC}),
\begin{equation}
\label{eq:CompareAv4}
\begin{split}
\left| (f)_{x_1,r_1} - (f)_{x_2,r_2} \right| \leq \frac{1}{6} C(n,\Lambda r_1,\Lambda \iota ) (\Lambda r_1)^{1/2}
\end{split}
\end{equation}

We are now ready to give the proof of Lemma \ref{le:Holder}.

\begin{proof}[Proof of Lemma \ref{le:Holder}]
First, we show that for $0<r<\iota/8$,
\[
|(f)_{x,r} - f(x)| \leq \frac{1}{3} C(n, \Lambda r, \Lambda \iota) (\Lambda r)^{1/2}.
\]
Set $s_i = r 2^{-i}$. 
Then, by (\ref{eq:CompareAv4}) we know that
\[
|(f)_{x,s_i} - (f)_{x,s_{i+1}} | \leq \frac{1}{6} C(n, \Lambda r, \Lambda \iota)(\Lambda r)^{1/2} 2^{-i/2}.
\]
Hence,
\[
|(f)_{x,s_i} - (f)_{x,s_j} | \leq \frac{1}{6} C(n, \Lambda r, \Lambda \iota) (\Lambda r)^{1/2}\left(2^{-(i-1)/2} - 2^{-(j-1)/2} \right).
\]
The claim follows by letting $j\to\infty$.

Now set $r_1 = r_2 = d(x_1,x_2)$. 
Then, again by (\ref{eq:CompareAv3}),
\[
|(f)_{x_1,r} - (f)_{x_2,r}| \leq \frac{1}{3} C(n, \Lambda r, \Lambda \iota) (\Lambda r)^{1/2}.
\]

The Lemma follows with the triangle inequality.
\end{proof}

\subsection{Distance function coordinates}

Let $(M,g) \in \mathcal{M}(n,\Lambda,\iota)$ and $p \in M$ as in Section \ref{se:HRsetup}. 
Pick an orthonormal basis $E_i$ of $T_pM$.
Define the points $p_i := \exp_p((\iota/4) E_i)$.
Set $\rho_i = \rho_{p_i} = d(p_i,.)$. 

\begin{Theorem}[Distance function coordinates]
\label{th:DistCoord}
If $r < \iota/64$ and
\begin{equation}
\label{eq:CondCoord}
C(n,\Lambda r, \Lambda \iota) (\Lambda r)^{1/2} < \frac{1}{2n}, 
\end{equation}
the functions $\rho_i$ ($i=1,\dots,n$) are coordinates on the ball $B_{r} (p)$, and in these coordinates the metric satisfies
\begin{align}
\label{eq:BilinearBounds}
(1 - n C(n, \Lambda r, \Lambda \iota) (\Lambda r)^{1/2}) \delta^{ij} \leq g^{ij} &\leq (1 + n C(n,\Lambda r, \Lambda \iota) (\Lambda r)^{1/2} ) \delta^{ij}, \\
\label{eq:HolderOnBall}
r^{1/2} [ g^{ij} ]_{C^{1/2}(B_{r}(p))} &\leq C(n,\Lambda r, \Lambda \iota) (\Lambda r)^{1/2}.
\end{align}
\end{Theorem}

\begin{proof}
Since
\[
g^{ij} = g(\nabla \rho_i, \nabla \rho_j),
\]
and $g^{ij}(p) = \delta^{ij}$, Gershgorin's circle theorem implies (\ref{eq:BilinearBounds}). 
Lemma \ref{le:Holder} implies (\ref{eq:HolderOnBall}).

It remains to check that the map $(\rho_1, \dots, \rho_n):B_r(p) \to \R^n$ is one-to-one.
Let $x_1, x_2 \in B_r(p)$, and let $q_1 = \gamma_{x_1,x_2}(\iota/4-\iota/64)$.
Note that since $r < \iota/64$, we have 
\[
\iota/4 - \iota/32 \leq d(p,q_1) \leq \iota/4.
\]
By (\ref{eq:BilinearBounds}), there is an $i_0$ such that
\[
|g(\nabla \rho_{q_1}, \nabla \rho_{i_0} )| \geq \sqrt{\frac{2}{n}}.
\]
The condition (\ref{eq:CondCoord}) implies by Lemma \ref{le:Holder} that the sign of $g(\nabla \rho_{q_1}, \nabla \rho_{i_0})(\gamma_{x_1,x_2}(s))$ is the same for all $0 \leq s < d(x_1,x_2)$.
Consequently, $\rho_{i_0}(x_1) \neq \rho_{i_0}(x_2)$.
\end{proof}

\subsection{Control of H\"{o}lder norms with higher exponents}

In the previous section, we obtained distance function coordinates for which the metric is controlled in a $C^{1/2}$ sense. 
In this section, we will use distance functions from additional points to control the H\"{o}lder norm $C^\alpha$ for arbitrary $0<\alpha<1$.

Let $q_1 \in \partial B_{\iota/4}(p)$ and define the function
\begin{equation}
\bar{\rho}_{q_1} := (\nabla \rho_{q_1}, \nabla \rho_j )(p) (\rho_j - \rho_j(p)) + \rho_{q_1}(p),
\end{equation}
where summation over $j=1,\dots, n$ is understood. 
We will explain why $\bar{\rho}_{q_1}$ is a good approximation to $\rho_{q_1}$. 

Let us first, however, recall a lemma on interior elliptic estimates (a simple consequence of \cite[Theorem 9.11]{gilbarg_elliptic_2001} and the Poincar\'{e} inequality) and phrase it in a dimensionless form that is useful for our applications.
\begin{Lemma}[Interior elliptic estimates]
\label{le:IntEll}
Suppose the function $u$ satisfies
\begin{equation}
(a^{ij} \partial_{x_i} \partial_{x_j} u + b^i \partial_{x^i}) u = f, \qquad \text{ on } B_r(0),
\end{equation}
where the coefficients satisfy
\begin{equation}
\frac{1}{K}\delta^{ij} \leq a^{ij} \leq K \delta^{ij}, \quad r^{1/2} [a^{ij}]_{C^{1/2}(B_r(0))} \leq K, \quad r |b_i| \leq K,
\end{equation}
then, for some constants $C_E(n,K)$ and $C_E(n,K,\alpha)$,
\begin{align}
\|\partial_{x_i} u\|_{L^{\infty}(B_{3r/4}(0))} &\leq C_E(n,K) \left(r \|f\|_{L^\infty(B_r(p))} + \frac{1}{r} \|u\|_{L^\infty(B_r(p))} \right)\\
r^{\alpha} [\partial_{x_i} u]_{C^\alpha(B_{3r/4}(0))} &\leq C_E(n, K, \alpha) \left(r \|f\|_{L^\infty(B_r(p))} + \frac{1}{r} \|u\|_{L^\infty(B_r(p))} \right)
\end{align}
\end{Lemma}

We will now use the H\"{o}lder bounds from the previous section to show a supremum bound on the difference $\rho_{q_1} - \bar{\rho}_{q_1}$.

\begin{Lemma}
\label{le:DiffRhoGuess}
Let $r < \iota/64$.
Then, the difference between $\rho_{q_1}$ and $\bar{\rho}_{q_1}$ satisfies
\begin{equation}
r^{-1} \| \rho_{q_1} - \bar{\rho}_{q_1} \|_{L^\infty(B_r(p))} \leq \frac{2}{3} (1 + \sqrt{\tilde{n}}) C(n, \Lambda r, \Lambda \iota) (\Lambda r)^{1/2},
\end{equation}
where $\tilde{n}$ is the number of $j = 1, \dots, n$ such that $g(\nabla \rho_j, \nabla \rho_{q_1})(p) \neq 0$.
\end{Lemma}

\begin{proof}
Consider a point $q \in B_r(p)$ and define $q_2 = \gamma_{p,q}(\iota/4)$.
Then 
\[
\begin{split}
\rho_{q_1}(q) - \rho_{q_1}(p) &= \int_0^{d(p,q)} g(\nabla \rho_{q_1}, \nabla \rho_{q_2} )(\gamma_{p,q}(s)) ds \\
& = \int_0^{d(p,q)} \left( g(\nabla \rho_{q_1}, \nabla \rho_{q_2} )(\gamma(s)) - g(\nabla \rho_{q_1},\nabla \rho_{q_2})(p) \right) ds
\\ & \quad+d(p,q) g(\nabla \rho_{q_1},\nabla \rho_{q_2})(p)
\end{split}
\]
We use that
\[
\begin{split}
\int_0^{d(p,q)} & |g(\nabla \rho_{q_1}, \nabla \rho_{q_2})(\gamma_{p,q}(s)) - g(\nabla \rho_{q_1},\nabla \rho_{q_2})(p)| ds \\
& \leq \int_0^{d(p,q)} C(n, \Lambda r, \Lambda \iota) \Lambda^{1/2} s^{1/2} ds \\
& \leq \frac{2}{3} C(n, \Lambda r, \Lambda \iota) \Lambda^{1/2} d(p,q)^{3/2}.
\end{split}
\]
In the same spirit, using that $g(\nabla \bar{\rho}_{q_1}, \nabla \rho_{q_2})(p) = g(\nabla \rho_{q_1},\nabla \rho_{q_2})(p)$, we find
\[
\begin{split}
\bar{\rho}_{q_1} (q) - \bar{\rho}_{q_1}(p) 
&= \int_0^{d(p,q)} g(\nabla \bar{\rho}_{q_1},\nabla \rho_{q_2})(\gamma_{p,q}(s))ds \\
&= \int_0^{d(p,q)} \left( g(\nabla \bar{\rho}_{q_1}, \nabla \rho_{q_2})(\gamma_{p,q}(s))- g(\nabla \bar{\rho}_{q_1}, \nabla \rho_{q_2})(p)\right) ds\\
&\quad + d(p,q) g(\nabla \rho_{q_1},\nabla \rho_{q_2})(p).
\end{split}
\]
We may now estimate
\[
\begin{split}
\int_0^{d(p,q)} & |g(\nabla \bar{\rho}_{q_1}, \nabla \rho_{q_2})(\gamma_{p,q}(s)) - g(\nabla \bar{\rho}_{q_1},\nabla \rho_{q_2})(p)| ds\\
&\leq \int_0^{d(p,q)} \left| g(\nabla \rho_{j}, \nabla \rho_{q_1})(p) (g(\nabla \rho_j, \nabla \rho_{q_2})(\gamma_{p,q}(s)) - g(\nabla \rho_j,\nabla \rho_{q_2})(p)) \right| ds \\
&\leq \sqrt{\tilde{n}} \frac{2}{3} C(n, \Lambda r, \Lambda \iota) \Lambda^{1/2} d(p,q)^{3/2},
\end{split}
\]
where $\tilde{n}$ is the number of $j$ such that $g(\nabla \rho_{j}, \nabla \rho_{q_1})(p) \neq 0$.
It follows that 
\[
\begin{split}
|\rho_{q_1}(q) - \bar{\rho}_{q_1}(q)| & 
= |\rho_{q_1}(q) - \rho_{q_1}(p) - (\bar{\rho}_{q_1}(q) - \bar{\rho}_{q_1}(p))| \\
& \leq \frac{2}{3} (1 + \sqrt{\tilde{n}}) C(n, \Lambda r, \Lambda \iota) \Lambda^{1/2} d(p,q)^{3/2}.
\end{split}
\]
\end{proof}

If we combine the supremum bound in Lemma \ref{le:DiffRhoGuess} of the difference between $\rho_{q_1}$ and $\bar{\rho}_{q_1}$ with the supremum bound on the Laplacian of a distance function (\ref{eq:BoundLaplRho}) and the interior elliptic estimates, we obtain the following lemma.

\begin{Lemma}
\label{le:BoundRhoMinGuess}
Let $r < \iota/64$ such that in addition
\begin{equation}
C(n, \Lambda r, \Lambda \iota) (\Lambda r)^{1/2} < \frac{1}{4n}.
\end{equation}
Then, the difference $\rho_{q_1} - \bar{\rho}_{q_1}$ satisfies
\begin{align*}
\left\|\frac{\partial \rho_{q_1}}{\partial \rho_i} - \frac{\partial \bar{\rho}_{q_1}}{\partial \rho_i}\right\|_{L^\infty(B_{r/2}(p))} & \leq  C_1(n, \Lambda r,\Lambda \iota) (\Lambda r)^{1/2}, \\
r^\alpha \left[ \frac{\partial \rho_{q_1}}{\partial \rho_i} - \frac{\partial \bar{\rho}_{q_1}}{\partial \rho_i} \right]_{C^\alpha(B_{r/2}(p))} & \leq C_1(n, \Lambda r, \Lambda \iota,\alpha)  (\Lambda r)^{1/2},
\end{align*}
where the constants $C_1(n, \Lambda r, \Lambda \iota)$ and $C_1(n, \Lambda r, \Lambda \iota, \alpha)$ follow from the proof below.
\end{Lemma}

\begin{proof}
We write out the left-hand side of the equation
\[
\Delta (\rho_{q_1} - \bar{\rho}_{q_1}) = \Delta \rho_{q_1} - \Delta \bar{\rho}_{q_1}
\]
in coordinates $\rho_i$ and get
\[
g^{ij} \partial_{\rho_i} \partial_{\rho_j} (\rho_{q_1} - \bar{\rho}_{q_1}) + \Delta \rho^i \partial_{\rho_i} (\rho_{q_1} - \bar{\rho}_{q_1}) = \Delta \rho_{q_1} - \Delta \bar{\rho}_{q_1}.
\]
By Theorem \ref{th:DistCoord}, with $Q = 4/3$,
\begin{align*}
\frac{1}{Q} \delta^{ij} \leq g^{ij} &\leq Q \delta^{ij} \\
r^{1/2} [g^{ij}]_{C^{1/2}(B_r(p))} &\leq \frac{1}{4n} (\Lambda r)^{1/2},
\end{align*}
and by the supremum bound on the Laplacian of distance functions, 
\[
r |\Delta \rho_i| \leq (n-1) \Lambda r \coth(\Lambda \iota/16).
\]
Write $\underline{\rho}:=(\rho_1, \dots, \rho_n)$ and $y_0 := \underline{\rho}(p)$.
Note that $B_{r/\sqrt{Q}}(y_0) \subset \underline{\rho}(B_r(p))$. 
Moreover, with some abuse of notation, for $y_1,y_2 \in B_{r/\sqrt{Q}}(y_0)$,
\[
\left( \frac{r}{\sqrt{Q}} \right)^{1/2} \frac{|g^{ij}(y_1)-g^{ij}(y_2)|}{|y_1-y_2|^{1/2}} \leq r^{1/2} [g^{ij}]_{C^{1/2}(B_r(p))}.
\]
We may therefore apply the elliptic estimates of Lemma \ref{le:IntEll} with 
\[
K := \max\left(Q, \frac{1}{4n}(\Lambda r)^{1/2}, (n-1) \Lambda r \coth(\Lambda \iota/16) \right),
\]
and obtain
\begin{equation}
\label{eq:ConsEll}
\begin{split}
& \| \partial_{\rho_i} (\rho_{q_1} - \bar{\rho}_{q_1}) \|_{L^\infty(B_{\lambda r/\sqrt{Q}}(y_0))}\\
& \leq C_E(n,K) \left(\left( \frac{r}{\sqrt{Q}}\right) \|\Delta (\rho_{q_1} - \bar{\rho}_{q_1})\|_{L^\infty(B_r(p))} 
 +\frac{\sqrt{Q}}{r}\|\rho_{q_1} - \bar{\rho}_{q_1} \|_{L^\infty(B_r(p))} \right),\\
& (r/\sqrt{Q})^\alpha [ \partial_{\rho_i}(\rho_{q_1} - \bar{\rho}_{q_1}) ]_{C^\alpha(B_{\lambda r/\sqrt{Q}}(y_0)))} \\
& \leq C_E(n,K,\alpha) \left(\left( \frac{r}{\sqrt{Q}}\right) \|\Delta (\rho_{q_1} - \bar{\rho}_{q_1})\|_{L^\infty(B_r(p))} 
 +\frac{\sqrt{Q}}{r}\|\rho_{q_1} - \bar{\rho}_{q_1} \|_{L^\infty(B_r(p))} \right).
\end{split}
\end{equation}

Now, we realize that the supremum bound (\ref{eq:BoundLaplRho}) also holds for $\Delta \rho_{q_1}$, and 
\[
\begin{split}
|\Delta \bar{\rho}| &= \left|\sum_{i=1}^n g(\nabla \rho_i, \nabla \rho_{q_1})(p) \Delta \rho_i \right| \\
& \leq \sqrt{\tilde{n}} (n-1) \Lambda \coth(\Lambda \iota/16).
\end{split}
\]
Finally, by Lemma \ref{le:DiffRhoGuess},
\[
r^{-1} \| \rho_{q_1} - \bar{\rho}_{q_1} \|_{L^\infty(B_r(p))} \leq \frac{2}{3} (1 + \sqrt{\tilde{n}}) C(n, \Lambda r, \Lambda \iota) (\Lambda r)^{1/2}
\]
so that the interior elliptic estimates (\ref{eq:ConsEll}) yield the result.
\end{proof}

\begin{Theorem}
Let $0 < \alpha <1$. If $r < \iota/64$ and 
\begin{equation}
C(n,\Lambda r, \Lambda \iota) (\Lambda r)^{1/2} < \frac{1}{4n}, 
\end{equation}
the functions $\rho_i$ ($i=1,\dots,n$) are coordinates on the ball $B_{r} (p)$, and in these coordinates the metric satisfies
\begin{equation}
(1 - n C(n, \Lambda r, \Lambda \iota) (\Lambda r)^{1/2}) \delta^{ij} \leq g^{ij} \leq (1 + n C(n,\Lambda r, \Lambda \iota) (\Lambda r)^{1/2} ) \delta^{ij}.
\end{equation}
If, moreover, with the constant $C_1$ as defined in Lemma \ref{le:BoundRhoMinGuess}
\begin{equation}
\label{eq:ControlRhoMinGuess}
C_1(n, \Lambda r, \Lambda \iota) (\Lambda r)^{1/2} < \frac{1}{8n^2},
\end{equation}
the metric coefficients additionally satisfy
\begin{equation}
\label{eq:HolderHigherAlpha}
r^\alpha [ g^{ij} ]_{C^\alpha(B_{r/2}(p))} \leq C_2(n,\Lambda r, \Lambda \iota,\alpha) (\Lambda r)^{1/2},
\end{equation}
for a constant $C_2(n, \Lambda r, \Lambda \iota, \alpha)$ that follows from the proof below.
\end{Theorem}

\begin{proof}
As in the proof of the harmonic radius estimate by Anderson and Cheeger \cite{anderson_compactness_1992}, we can introduce more points $p_{k\ell}$, $k,\ell = 1, \dots, n$, $\ell > k$ given by 
\[
p_{k\ell} := \exp_p\left( \frac{E_k + E_\ell}{\sqrt{2}} \frac{\iota}{4} \right).
\]
For every such $k$ and $\ell$, we write out in distance function coordinates the equations $|\nabla \rho_{k\ell}|^2 = 1$, that is
\[
\sum_{1 \leq i, j \leq n} g^{ij} \frac{\partial \rho_{k\ell}}{\partial \rho_i} \frac{\partial \rho_{k\ell}}{\partial \rho_j} = 1.
\]
We use that $g^{ii} = 1$ and that $g^{ij} = g^{ji}$, so that
\[
\label{eq:Systemg}
2 \sum_{1 \leq i < j \leq n} g^{ij} \frac{\partial \rho_{k\ell} }{\partial \rho_i}\frac{\partial \rho_{k\ell}}{\partial \rho_j} 
= 1 - \sum_{1 \leq i \leq n} \left(\frac{\partial \rho_{k\ell}}{\partial \rho_i}\right)^2.
\]
Note that this is a linear system for the coefficients $g^{ij}$, $1 \leq i < j \leq n$.
Set 
\[
\bar{\rho}_{kl} = \frac{1}{\sqrt{2}} (\rho_k- \rho_k(p)) + \frac{1}{\sqrt{2}} (\rho_\ell - \rho_\ell(p)) + \rho_{kl}(p).
\]
By Lemma \ref{le:BoundRhoMinGuess} we know that
\[
\left\| \frac{\partial \rho_{kl}}{\partial \rho_i} - \frac{\partial \bar{\rho}_{kl}}{\partial \rho_i} \right\|_{B_{r/2}(p)}
\leq C_1(n, \Lambda r, \Lambda \iota) (\Lambda r)^{1/2},
\]
so that by the condition (\ref{eq:ControlRhoMinGuess}), we conclude
\[
\left| \frac{\partial \rho_{kl}}{\partial \rho_i} - \frac{1}{\sqrt{2}} \right| \leq \frac{1}{8n^2},
\]
if $i = k$ or $i = \ell$, and otherwise
\[
\left| \frac{\partial \rho_{kl}}{\partial \rho_i}\right| \leq \frac{1}{8n^2}.
\]
It follows that
\[
2 \sum_{(i,j)\neq (k,\ell)} \left|\frac{\partial \rho_{k\ell}}{\partial \rho_i} \frac{\partial \rho_{k\ell}}{\partial \rho_j}\right| + \frac{1}{2} \leq 2 \frac{\partial \rho_{k\ell}}{\partial \rho_k} \frac{\partial \rho_{k\ell}}{\partial \rho_\ell}.
\]
This guarantees that the determinant appearing in the system for $g^{ij}$ (\ref{eq:Systemg}) is larger than or equal to $2^{-n}$, or, alternatively, that (\ref{eq:Systemg}) can be solved by $L U$ decomposition where every diagonal element is larger than or equal to $1/2$. 
This implies the bound (\ref{eq:HolderHigherAlpha}).
\end{proof}

\subsection{Construction of harmonic coordinates}
\label{se:ConstrHarm}

With the distance functions as coordinates at hand, we may now construct harmonic coordinate functions by solving a Dirichlet problem.

Define harmonic functions $b_i: B_r(p) \to \R$ by 
\begin{equation}
\begin{cases}
\Delta b_i = 0, & \text{ on } B_r(p),\\
b_i  = \rho_i, & \text{ on } \partial B_r(p).
\end{cases}
\end{equation}

In Theorem \ref{th:HarmCoord} below we will show that on a smaller ball, the $b_i$ are coordinate functions with bounds on the H\"{o}lder norm of the metric coefficients.
Now we have constructed the distance function coordinates we may use interior elliptic estimates in these coordinates.
It is an important realization that such estimates yield control due to the supremum bound on the Laplacian of the distance functions (\ref{eq:BoundLaplRho}).
To exploit the interior estimates, we will also need to get control on the supremum norm of $b_i - \rho_i$. 
This is facilitated by exploiting quantitative versions of the maximum principle.

\begin{Theorem}
\label{th:HarmCoord}
Let $0< \alpha < 1$. If $r < \iota/64$ and
\begin{equation}
C_5(n, \Lambda r, \Lambda \iota) (\Lambda r)^{1/2} \leq \frac{1}{n}
\end{equation}
then the functions $b_i$ constructed above are harmonic coordinates on the ball $B_{r/2}(p)$, and in these coordinates the metric satisfies
\begin{align}
( 1 - n C_5(n,\Lambda r, \Lambda \iota) (\Lambda r)^{1/2}) \delta^{ij} \leq g^{ij} &\leq (1 + n C_5(n,\Lambda r, \Lambda \iota) (\Lambda r)^{1/2} ) \delta^{ij}, \\
r^\alpha [g^{ij}]_{C^\alpha(B_{r/2}(p))} &\leq C_7(n,\Lambda r, \Lambda \iota, \alpha) (\Lambda r)^{1/2},
\end{align}
where the constants $C_5(n,\Lambda r, \Lambda \iota)$ and $C_7(n,\Lambda r, \Lambda \iota,\alpha)$ follow from the proof below.
\end{Theorem}

\begin{proof}
We first require that 
\[
C(n,\Lambda r, \Lambda \iota) (\Lambda r)^{1/2} < \frac{1}{4n},
\]
so that the $\rho_i$ are coordinates on the ball $B_{r}(p)$.

Recall the bound 
\[
\label{eq:SupDeltaRho}
| \Delta \rho_i | \leq (n-1) \Lambda \coth(\Lambda \rho_i).
\]
Following notation from \cite{cheeger_degeneration_2001}, we introduce the function $\underline{L}_R(r)$ that was also used by Abresch and Gromoll \cite{abresch_complete_1990},
\begin{equation}
\label{eq:DefAbreschGromoll}
\underline{L}_R(r) = \int_r^R \int_s^R \frac{\sinh^{n-1}(\Lambda \tau)}{\sinh^{n-1}(\Lambda s )} d\tau ds.
\end{equation}
In the simply connected model space of constant sectional curvature $-\Lambda^2$, it holds that $\underline{\Delta}\, \underline{L}_R\equiv 1$, and $\underline{L}_R'(r) < 0$, for $0 < r < R$.
Now we choose $R = \iota / 4 + r$. 
By Laplacian comparison, the function $u := \underline{L}_{\iota/4 + r}(d(.,p_i))$ satisfies
\[
\Delta u \geq 1.
\]
Therefore,
\[
\Delta \left[ b_i - \rho_i \pm (n-1) \Lambda \coth(\Lambda \iota/16) u \right] \gtreqless 0.
\]
From the expression (\ref{eq:DefAbreschGromoll}) it follows that there exists a $C_3(n,\Lambda r, \Lambda \iota)$ such that
\[
r^{-1} \| u\|_{L^\infty(B_r(p))}  \leq \frac{1}{(n-1) \Lambda \coth(\Lambda \iota / 16)} C_3(n,\Lambda r, \Lambda \iota) \Lambda r,
\]
so that by the maximum principle the following supremum bound holds
\[
r^{-1} \| b_i - \rho_i \|_{L^\infty(B_r(p))} \leq C_3(n, \Lambda r,\Lambda \iota ) \Lambda r.
\]
In the coordinates $\rho_i$, the Laplacian of $b_i - \rho_i$ is expressed as 
\[
g^{kl} \frac{\partial}{\partial \rho_k} \frac{\partial}{\partial\rho_l}
  (b_i - \rho_i)
  - (\Delta \rho_k) \frac{\partial }{\partial \rho_k} 
  (b_i - \rho_i) 
= - \Delta \rho_i.
\]
Note that by (\ref{eq:SupDeltaRho}), for $k = 1, \dots, n$,
\[
r|\Delta \rho_k | \leq (n-1) (\Lambda r) \coth(\Lambda \iota/16).
\]
Consequently, from interior elliptic estimates (see Lemma \ref{le:IntEll} or \cite[Theorem 9.11]{gilbarg_elliptic_2001}) we obtain
\[
\left\| \frac{\partial b_i}{\partial \rho_j} - \delta_{ij} \right\|_{L^\infty(B_{r/2}(p))} \leq C_4(n, \Lambda r, \Lambda \iota) (\Lambda r).
\]
Hence, the map $p \to (b_1(p), \dots, b_n(p))$ is one-to-one if
\[
C_4(n,\Lambda r, \Lambda \iota)(\Lambda r) < \frac{1}{n}.
\]
Since
\begin{equation}
\label{eq:ExpandIPbi}
g(\nabla b_i, \nabla b_j) = g(\nabla \rho_k, \nabla \rho_\ell) \frac{\partial b_i}{\partial \rho_k} \frac{\partial b_j}{\partial \rho_\ell},
\end{equation}
we may estimate
\[
\left\|g(\nabla b_i, \nabla b_j) - \delta_{ij}\right\|_{L^\infty(B_{r/2})} \leq C_5(n, \Lambda r, \Lambda \iota) (\Lambda r)^{1/2}.
\]

From the interior elliptic estimates it also follows that for $x_1, x_2 \in B_{r/2}(p)$,
\[
r^\alpha  \left| \frac{\partial b_i}{\partial \rho_j}(x_1) - \frac{\partial b_i}{\partial \rho_j}(x_2) \right| 
\leq C_6(n, \Lambda r, \Lambda \iota,\alpha) (\Lambda r) d(x_1,x_2)^\alpha.
\]
Again using (\ref{eq:ExpandIPbi}) 
we find
\[
r^\alpha \left| g(\nabla b_i, \nabla b_j)(x_1) - g(\nabla b_i, \nabla b_j)(x_2) \right| \leq C_7(n,\Lambda r,\Lambda \iota,\alpha) (\Lambda r)^{1/2} d(x_1,x_2)^\alpha.
\]
\end{proof}

\bibliography{ReferencesEigenfunctions}
\bibliographystyle{plain}

\end{document}